\documentclass[12pt,english]{article}
\pdfoutput=1
\usepackage[T1]{fontenc}
\usepackage[latin9]{inputenc}
\usepackage{geometry}
\usepackage{amsmath, amsthm, amssymb}
\usepackage{mathtools}
\usepackage{color}
\usepackage{array}
\usepackage{multirow, multicol}
\usepackage{float}
\usepackage{bm}
\usepackage{tikz}
\usetikzlibrary{arrows, shapes}
\usepackage{hyperref}
\usepackage{caption}
\usepackage{microtype}
\usepackage{titlesec}

\providecommand{\tabularnewline}{\\}

\numberwithin{equation}{section}
\theoremstyle{plain}
\newtheorem{thm}{Theorem}[section]
\theoremstyle{plain}
\newtheorem{lem}[thm]{Lemma}
\theoremstyle{plain}
\newtheorem{cor}[thm]{Corollary}
\theoremstyle{plain}
\newtheorem{prop}[thm]{Proposition}
\theoremstyle{plain}
\newtheorem{conjecture}[thm]{Conjecture}
\theoremstyle{definition}
\newtheorem{problem}[thm]{Problem}
\theoremstyle{plain}
\newtheorem{question}[thm]{Question}

\DeclareMathOperator{\Des}{Des}
\DeclareMathOperator{\pk}{pk}
\DeclareMathOperator{\des}{des}
\DeclareMathOperator{\st}{st}
\DeclareMathOperator{\bdes}{bdes}
\DeclareMathOperator{\Bdes}{Bdes}
\DeclareMathOperator{\rbdes}{rbdes}
\DeclareMathOperator{\LRmax}{LRmax}
\DeclareMathOperator{\RLmax}{RLmax}
\DeclareMathOperator{\run}{run}
\DeclareMathOperator{\occ}{occ}

\DeclareMathOperator{\lddes}{lddes}
\DeclareMathOperator{\sdes}{sdes}
\DeclareMathOperator{\fix}{fix}
\DeclareMathOperator{\std}{std}
\DeclareMathOperator{\hibasc}{hibasc}
\DeclareMathOperator{\lobasc}{lobasc}
\DeclareMathOperator{\ini}{ini}
\DeclareMathOperator{\con}{con}
\DeclareMathOperator{\core}{core}

\newcommand{\MM}{\mathcal{M}^{(2)}}

\titleformat{\section}
  {\normalfont\large\bfseries}{\thesection.}{.8ex}{} 
\titleformat{\subsection}
  {\normalfont\bfseries}{\thesubsection.}{.8ex}{} 

\tikzstyle{pathdefault}=[draw, line width=1, solid, color=black]
\tikzstyle{nodedefault}=[circle, inner sep=1.1, fill=black]
\tikzstyle{nodered}=[circle, inner sep=1.1, fill=red]
\tikzstyle{nodeblue}=[circle, inner sep=1.1, fill=blue]
\tikzstyle{empty}=[]
\tikzstyle{nodeellipsis}=[circle, inner sep=0.5, fill=black]
\tikzstyle{pathcolor1}=[draw, line width=1, solid, color=red]
\tikzstyle{pathcolor2}=[draw, line width=1, solid, color=blue]
\tikzstyle{pathcolorlight}=[draw, line width=1, dotted, color=lightgray]
\tikzstyle{arbpathcolor0}=[line width=1, dashdotted, color=black]
\tikzstyle{arbpathcolor1}=[line width=1, densely dashed, color=red]
\tikzstyle{arbpathdefault}=[line width=1, densely dotted, color=blue]

\newcounter{id}
\newcommand{\drawlinedotswithstyle}[4]{
 \def\x{{#3}}
 \def\y{{#4}}
 \tikzstyle{thispathstyle}=[#1]
 \tikzstyle{thisnodestyle}=[#2]
 \setcounter{id}{-1}
 \foreach \j in {#3}{\stepcounter{id}}
 \foreach \i in {1,...,\the\value{id}}{
  \path[thispathstyle] (\x[\i],\y[\i]) --(\x[\i-1],\y[\i-1]);
 }
 \foreach \i in {1,...,\the\value{id}}{
  \node[thisnodestyle] at (\x[\i],\y[\i]) {};
 }
 \node[thisnodestyle] at (\x[0],\y[0]) {};
}

\DeclareDocumentCommand{\drawlinedots}{ O{pathdefault} O{nodedefault} m m}{\drawlinedotswithstyle{#1}{#2}{#3}{#4}}
\DeclareDocumentCommand{\drawlinedotsred}{ O{pathcolor1} O{nodered} m m}{\drawlinedotswithstyle{#1}{#2}{#3}{#4}}
\DeclareDocumentCommand{\drawlinedotsblue}{ O{pathcolor2} O{nodeblue} m m}{\drawlinedotswithstyle{#1}{#2}{#3}{#4}}
\DeclareDocumentCommand{\drawlinedotsredpt}{ O{pathcolor1} O{nodedefault} m m}{\drawlinedotswithstyle{#1}{#2}{#3}{#4}}
\DeclareDocumentCommand{\drawlinedotsbluept}{ O{pathcolor2} O{nodedefault} m m}{\drawlinedotswithstyle{#1}{#2}{#3}{#4}}

\let\originalleft\left
\let\originalright\right
\renewcommand{\left}{\mathopen{}\mathclose\bgroup\originalleft}
\renewcommand{\right}{\aftergroup\egroup\originalright}

\usepackage{authblk}
\title{Counting pattern-avoiding permutations by big descents}
\author[1]{Sergi Elizalde\thanks{\tt{sergi.elizalde@dartmouth.edu}}} 
\author[2]{Johnny Rivera, Jr.\thanks{\tt{jorivera@vt.edu}}} 
\author[3]{Yan Zhuang\thanks{\tt{yazhuang@davidson.edu}}}
\affil[1]{Department of Mathematics, Dartmouth College}
\affil[2]{Department of Mathematics, Virginia Tech} 
\affil[3]{Department of Mathematics and Computer Science, Davidson College}
\date{\today}

\newtheorem*{lemG}{Lemma \ref{lem:G}}
\newtheorem*{lemFG}{Lemma \ref{lem:FG}}
\begin{document}

\maketitle

\begin{abstract}
A descent $k$ of a permutation $\pi=\pi_{1}\pi_{2}\dots\pi_{n}$ is called a \textit{big descent} if $\pi_{k}>\pi_{k+1}+1$; denote the number of big descents of $\pi$ by $\bdes(\pi)$. We study the distribution of the $\bdes$ statistic over permutations avoiding prescribed sets of length-three patterns. Specifically, we classify all pattern sets $\Pi\subseteq\mathfrak{S}_{3}$ of size 1 and 2 into $\bdes$-Wilf equivalence classes, and we derive a formula for the distribution of big descents for each of these classes. Our methods include generating function techniques along with various bijections involving objects such as Dyck paths and binary words. Several future directions of research are proposed, including conjectures concerning real-rootedness, log-concavity, and Schur positivity.
\end{abstract}
\textbf{\small{}Keywords: }{\small{}permutation patterns, big descents, $\st$-Wilf equivalence, Dyck paths, binary words, Narayana numbers}{\let\thefootnote\relax}

\tableofcontents{}

\section{Introduction}

Let $\mathfrak{S}_{n}$ denote the symmetric group of permutations on $[n]\coloneqq\{1,2,\dots,n\}$. We treat permutations as words on distinct letters by writing them in one-line notation; that is, if $\pi\in\mathfrak{S}_{n}$, then we write $\pi=\pi_{1}\pi_{2}\dots\pi_{n}$ where $\pi_{k}=\pi(k)$ for all $k\in[n]$. We take $\mathfrak{S}_{0}\coloneqq\{\varepsilon\}$ where $\varepsilon$ is the empty word. The length of a word $w$ is denoted by $\left|w\right|$; this includes when $w$ is a permutation, so that $\left|\pi\right|=n$ whenever $\pi\in\mathfrak{S}_{n}$.

If $w$ is a word consisting of $n$ distinct positive integers, then we let $\std(w)$ denote the permutation in $\mathfrak{S}_{n}$ obtained by replacing the smallest letter of $w$ by 1, the second smallest by 2, and so on. The permutation $\std(w)$ is called the \textit{standardization} of $w$. For example, the standardization of $5718$ is $\std(5718)=2314$.

An \textit{occurrence} of $\sigma\in\mathfrak{S}_{k}$ in $\pi\in\mathfrak{S}_{n}$ is a subsequence $w$ of $\pi$ whose standardization is $\sigma$. For example, $527$ is an occurrence of $213$ in $1523764$, but $7612543$ contains no occurrences of $213$. We say that $\pi$ \textit{avoids} $\sigma$ (as a \textit{pattern}), or that $\pi$ is \textit{$\sigma$-avoiding}, if $\pi$ contains no occurrences of $\sigma$. Let $\mathfrak{S}_{n}(\sigma)$ denote the subset of permutations in $\mathfrak{S}_{n}$ which avoid $\sigma$. So, as noted above, $7612543\in\mathfrak{S}_{7}(213)$. More generally, given a set $\Pi$ of patterns, we let $\mathfrak{S}_{n}(\Pi)$ denote the subset of permutations in $\mathfrak{S}_{n}$ which avoid every pattern in $\Pi$.\footnote{For ease of notation, given a specific $\Pi$, we will often omit the curly braces enclosing the set $\Pi$ when writing out $\mathfrak{S}_{n}(\Pi)$ and in similar notations.}

\subsection{History and background}

Permutation patterns have received widespread attention since the seminal work of Simion and Schmidt \cite{Simion1985}. An important theme in the enumerative study of permutation patterns is the classification of pattern sets into Wilf equivalence classes. Two pattern sets $\Pi$ and $\Pi^{\prime}$ are called \textit{Wilf equivalent} if $\left|\mathfrak{S}_{n}(\Pi)\right|=\left|\mathfrak{S}_{n}(\Pi^{\prime})\right|$ for all $n\geq0$. For example, it is well known that all length-three singletons belong to the same Wilf equivalence class, whose common enumeration is given by the Catalan numbers. Simon and Schmidt \cite{Simion1985} studied the enumeration of $\mathfrak{S}_{n}(\Pi)$ for all sets $\Pi\subseteq\mathfrak{S}_{3}$.

Compared to the study of permutation patterns, a more classical problem within permutation enumeration is that of counting permutations with respect to permutation statistics. One such statistic is the descent number $\des$. We say that $k\in[n-1]$ is a \textit{descent} of $\pi\in\mathfrak{S}_{n}$ if $\pi_{k}>\pi_{k+1}$, and $\des(\pi)$ is defined to be the number of descents of $\pi$. The study of the descent number dates back to MacMahon \cite{macmahon}, and there is now an immense literature devoted to the \textit{Eulerian polynomials} $A_{n}(t)\coloneqq\sum_{\pi\in\mathfrak{S}_{n}}t^{\des(\pi)}$ and their coefficients, the \textit{Eulerian numbers}. See \cite{Petersen2015} for an introduction.

Given a pattern set $\Pi$ and a permutation statistic $\st$, it is natural to ask for the distribution of $\st$ over $\mathfrak{S}_{n}(\Pi)$. For example, Simion \cite{Simion1994} used a bijection with noncrossing partitions to show that the distribution of $\des$ over $\mathfrak{S}_{n}(132)$ is given by the Narayana numbers. By symmetry arguments, the same distribution is achieved over $\mathfrak{S}_{n}(213)$, $\mathfrak{S}_{n}(231)$, and $\mathfrak{S}_{n}(312)$. More recently, Barnabei, Bonetti, and Silimbani \cite{Barnabei2010} employed a bijection with Dyck paths to derive the generating function 
\begin{equation}
\sum_{n=0}^{\infty}\sum_{\pi\in\mathfrak{S}_{n}(123)}t^{\des(\pi)}x^{n}=\frac{1-2t(1-t)x-2t(1-t)^{2}x^{2}-\sqrt{1-4tx(1+(1-t)x)}}{2t^{2}x(1+(1-t)x)}\label{e-123des}
\end{equation}
for the distribution of $\des$ over $\mathfrak{S}_{n}(123)$. The corresponding generating function for $\mathfrak{S}_{n}(321)$ is readily obtained from \eqref{e-123des}.

The refined enumeration of pattern-avoiding permutations by statistics naturally leads to the notion of $\st$-Wilf equivalence, defined as follows: given a statistic $\st$, two pattern sets $\Pi$ and $\Pi^{\prime}$ are said to be \textit{$\st$-Wilf equivalent} if, for all $n\geq0$, the distribution of $\st$ over $\mathfrak{S}_{n}(\Pi)$ is equal to the distribution of $\st$ over $\mathfrak{S}_{n}(\Pi^{\prime})$. When $\st$ is non-negative integer-valued, this amounts to the equality of polynomials 
\[
\sum_{\pi\in\mathfrak{S}_{n}(\Pi)}t^{\st(\pi)}=\sum_{\pi\in\mathfrak{S}_{n}(\Pi^{\prime})}t^{\st(\pi)}
\]
for all $n\geq0$. As far as we are aware, the idea of refining Wilf equivalence by a permutation statistic first appeared in the work of Robertson, Saracino, and Zeilberger \cite{Robertson2002}, who studied $\fix$-Wilf equivalence where $\fix$ is the number of fixed points. However, the general notion of $\st$-Wilf equivalence was introduced later by Dokos, Dwyer, Johnson, Sagan, and Selsor \cite{Dokos2012}, who classified all $\Pi\subseteq\mathfrak{S}_{3}$ into $\st$-Wilf equivalence classes for the major index and inversion number statistics.

Another permutation statistic\textemdash one which has received considerably less attention than the statistics already mentioned\textemdash is the $r$-descent number. Given a non-negative integer $r$, we say that $k\in[n-1]$ is an \textit{$r$-descent} of $\pi\in\mathfrak{S}_{n}$ if $\pi_{k}>\pi_{k+1}+r$, and we let $\des_{r}(\pi)$ denote the number of $r$-descents of $\pi$. Setting $r=0$ recovers the classical descent number. For example, given $\pi=7421365$, we have $\des_{0}(\pi)=\des(\pi)=4$, $\des_{1}(\pi)=2$, and $\des_{2}(\pi)=1$. Riordan \cite{Riordan1958} and Foata\textendash Sch{\"u}tzenberger \cite{Foata1970} proved a number of formulas for the $r$\textit{-Eulerian polynomials} $^{r\!\!}A_{n}(t)\coloneqq\sum_{\pi\in\mathfrak{S}_{n}}t^{\des_{r}(\pi)}$, some of which can also be derived using a ``relaxed'' version of $P$-partitions due to Petersen and the third author \cite{Petersen}. Moreover, Shareshian and Wachs \cite[Section 9.1]{Shareshian2016} defined a generalized $q$-Eulerian polynomial for incomparability graphs of posets, which agrees with $^{r\!\!}A_{n}(t)$ as a special case. They showed that these generalized $q$-Eulerian polynomials can be obtained as certain specializations of chromatic quasisymmetric functions. Examples of known formulas for the $r$-Eulerian polynomials include the recurrence
\[
^{r\!\!}A_{n}(t)=(r+1+(n-r-1)t)\:{}^{r\!\!}A_{n-1}(t)+t(1-t)\:{}^{r\!\!}A_{n-1}^{\prime}(t)
\]
for all $n\geq r+1$, and the generalized Carlitz identity
\[
\frac{^{r\!\!}A_{n+r}(t)}{(r+1)!\,(1-t)^{n+1+r}}=\sum_{k=0}^{\infty}(k+1+r)^{n-1}{k+1+r \choose r+1}t^{k}.
\]
However, much less seems to be known about the distribution of $r$-descents over restricted families of permutations.

\subsection{Overview of results}

Let us call $k\in[n-1]$ a \textit{big descent} of $\pi\in\mathfrak{S}_{n}$ if $\pi_{k}>\pi_{k+1}+1$\textemdash in other words, if $k$ is a $1$-descent of $\pi$\textemdash and let $\bdes(\pi)\coloneqq\des_{1}(\pi)$ denote the number of big descents of $\pi$. In this paper, we study the distribution of big descents over sets of pattern-avoiding permutations; specifically, permutations avoiding a single pattern of length 3 or two patterns of length 3. We end this introduction by summarizing our main results and outlining the structure of this paper.

The \textit{reverse} $\pi^{r}$ and the \textit{complement} $\pi^{c}$ of a permutation $\pi\in\mathfrak{S}_{n}$ are defined by
\[
\pi^{r}\coloneqq\pi_{n}\pi_{n-1}\dots\pi_{1}\quad\text{and}\quad\pi^{c}\coloneqq(n+1-\pi_{1})\;(n+1-\pi_{2})\;\dots\;(n+1-\pi_{n}),
\]
respectively. In addition, the \textit{reverse-complement} $\pi^{rc}$ of $\pi$ is defined by 
\[
\pi^{rc}\coloneqq(\pi^{r})^{c}=(\pi^{c})^{r}.
\]
For example, if $\pi=1425736$, then we have $\pi^{r}=6375241$, $\pi^{c}=7463152$, and $\pi^{rc}=2513647$. Given a set $\Pi$ of permutations, let
\[
\Pi^{rc}\coloneqq\{\,\pi^{rc}\colon\pi\in\Pi\,\}.
\]
Note that $\pi\in\mathfrak{S}_{n}(\Pi)$ if and only if $\pi^{rc}\in\mathfrak{S}_{n}(\Pi^{rc})$, and that the number of big descents is invariant under reverse-complementation; therefore, $\Pi$ and $\Pi^{rc}$ are $\bdes$-Wilf equivalent for all pattern sets $\Pi$. We refer to such equivalences as \textit{trivial}.

Among our primary contributions is a full $\bdes$-Wilf equivalence classification for all $\Pi\subseteq\mathfrak{S}_{3}$ with $\left|\Pi\right|=1$ or $\left|\Pi\right|=2$. Below, $[\Pi]_{\bdes}$ denotes the $\bdes$-Wilf equivalence class of $\Pi$.

\begin{thm} \label{t-bdesW1}
The following are all $\bdes$-Wilf equivalence classes among single patterns of length 3\textup{:}
\begin{align*}
[231]_{\bdes} & =\{231,312\}, & [132]_{\bdes} & =\{132,213\},\\
[123]_{\bdes} & =\{123\},\quad\text{and} & [321]_{\bdes} & =\{321\}.
\end{align*}
In particular, there are no non-trivial $\bdes$-Wilf equivalences among these patterns.
\end{thm}

\begin{thm} \label{t-bdesW2}
The following are all $\bdes$-Wilf equivalence classes among pattern sets $\Pi\subseteq\mathfrak{S}_{3}$ with $\left|\Pi\right|=2$\textup{:}
\begin{align*}
[213,231]_{\bdes} & =\left\{ \{213,231\},\{132,312\},\{213,312\},\{132,231\} \right\},\\
[231,321]_{\bdes} & =\left\{ \{231,321\},\{312,321\} \right\},\\
[123,231]_{\bdes} & =\left\{ \{123,231\},\{123,312\} \right\},\\
[132,321]_{\bdes} & =\left\{ \{132,321\},\{213,321\} \right\},\\
[123,132]_{\bdes} & =\left\{ \{123,132\},\{123,213\},\{132,213\} \right\},\\
[123,321]_{\bdes} & =\left\{ \{123,321\} \right\},\quad\text{and}\\
[231,312]_{\bdes} & =\left\{ \{231,312\} \right\}.
\end{align*}
In particular, the classes $[213,231]_{\bdes}$ and $[123,132]_{\bdes}$ contain non-trivial $\bdes$-Wilf equivalences.
\end{thm}

\renewcommand{\arraystretch}{1.5}
\begin{table}
\begin{centering}
\begin{tabular}{|>{\raggedright}m{0.75in}|>{\raggedright}m{3.65in}|>{\raggedright}m{0.75in}|>{\raggedright}m{0.6in}|}
\hline 
$\Pi$ & Distribution of $\bdes$ over $\mathfrak{S}_{n}(\Pi)$ & Reference & OEIS\tabularnewline
\hline 
$\{132\}$ & $B(t,x;132)=$

$\frac{1-2(1-t)x+(1-t)(1-2t)x^{2}-\sqrt{1-4x+6(1-t)x^{2}-4(1-t)^{2}x^{3}+(1-t)^{2}x^{4}}}{2tx(1-(1-t)x)}$ & Thm.\ \ref{t-132gf} & none\tabularnewline
\hline 
$\{231\}$ & $b_{n,k}(231)=\frac{2^{n-2k-1}}{k+1}{n-1 \choose 2k}{2k \choose k}$ & Cor.\ \ref{c-231} & A091894\tabularnewline
\hline 
$\{321\}$ & $B(t,x;321)=\frac{2}{1-2(1-t)x^{2}+\sqrt{1-4x+4(1-t)x^{2}}}$ & Thm.\ \ref{t-321gf} & none\tabularnewline
\hline 
$\{123\}$ & $b_{n,k}(123)=\frac{2}{n+1}{n+1 \choose k+2}{n-2 \choose k}$ & Thm.\ \ref{t-123} & A108838\tabularnewline
\hline 
$\{213,231\}$ & \multirow{2}{3.65in}{$b_{n,k}(213,231)={n \choose 2k+1}$} & Thm.\ \ref{t-213-231} & A034867\tabularnewline
\cline{1-1} \cline{3-4} \cline{4-4} 
$\{213,312\}$ &  & Thm.\ \ref{t-213-312} & A034867\tabularnewline
\hline 
$\{123,132\}$ & \multirow{2}{3.65in}{$B(t,x;123,132)=\frac{1-x+(1-t)x^{2}-(1-t)^{2}x^{3}}{1-2x+(1-t)x^{2}+t(1-t)x^{3}}$} & Thm.\ \ref{t-123-132} & none\tabularnewline
\cline{1-1} \cline{3-4} \cline{4-4} 
$\{132,213\}$ &  & Thm.\ \ref{t-132-213} & none\tabularnewline
\hline 
$\{231,321\}$ & $B(t,x;231,321)=\frac{1-x}{1-2x+(1-t)x^{3}}$ & Thm.\ \ref{t-231-321} & A334658\tabularnewline
\hline 
$\{123,231\}$\vspace{4bp}
 & \vspace{2bp}
$b_{n,k}(123,231)=\begin{cases}
n, & \text{if }k=0,\\
{n-1 \choose 2}, & \text{if }k=1,\\
0, & \text{otherwise.}
\end{cases}$\vspace{2bp}
 & Prop.\ \ref{p-123-231}\vspace{4bp}
 & none\vspace{4bp}
\tabularnewline
\hline 
$\{132,321\}$\vspace{4bp}
 & \vspace{2bp}
$b_{n,k}(132,321)=\begin{cases}
2, & \text{if }k=0,\\
{n \choose 2}-1, & \text{if }k=1,\\
0, & \text{otherwise.}
\end{cases}$\vspace{2bp}
 & Prop.\ \ref{p-132-321}\vspace{4bp}
 & none\vspace{4bp}
\tabularnewline
\hline 
$\{231,312\}$\vspace{4bp}
 & \vspace{2bp}
$b_{n,k}(231,312)=\begin{cases}
2^{n-1}, & \text{if }k=0,\\
0, & \text{otherwise.}
\end{cases}$\vspace{2bp}
 & Prop.\ \ref{p-231-312}\vspace{4bp}
 & none\vspace{4bp}
\tabularnewline
\hline 
$\{123,321\}$ & $B(t,x;123,321)=$

$\qquad\qquad1+x+2x^{2}+(2+2t)x^{3}+(1+2t+t^{2})x^{4}$ & Prop.\ \ref{p-123-321} & none\tabularnewline
\hline 
\end{tabular}
\par\end{centering}
\caption{\label{tb-enum}Summary of enumeration formulas.}
\end{table}

Let $b_{n,k}(\Pi)$ denote the number of permutations $\pi\in\mathfrak{S}_{n}(\Pi)$ with $\bdes(\pi)=k$, let
\[
B_{n}(t;\Pi)\coloneqq\sum_{\pi\in\mathfrak{S}_{n}(\Pi)}t^{\bdes(\pi)}=\sum_{k\geq0}b_{n,k}(\Pi)t^{k}
\]
denote the polynomial encoding the distribution of $\bdes$ over $\mathfrak{S}_{n}(\Pi)$, and let
\[
B(t,x;\Pi)\coloneqq\sum_{n=0}^{\infty}B_{n}(t;\Pi)x^{n}
\]
denote the ordinary generating function for the polynomials $B_{n}(t;\Pi)$. We derive a formula for the big descent distribution over each $\bdes$-Wilf equivalence class listed in Theorems~\ref{t-bdesW1}\textendash \ref{t-bdesW2}. A summary of these formulas is given in Table \ref{tb-enum}, including references to corresponding entries in the On-Line Encyclopedia of Integer Sequences (OEIS) \cite{oeis}. Most of our proofs involve bijections with Dyck paths and binary words, which translate big descents into occurrences of certain factors in an associated path or word. The distribution of $\bdes$ over $123$-avoiding permutations is especially noteworthy, as it is given by a variation of the Narayana numbers, yet we do not know of a simple bijective proof of this fact; our proof is quite convoluted, requiring a combination of generating function operations, some auxiliary statistics, and a bijection involving colored Motzkin paths.

The organization of this paper is as follows. Section \ref{s-single} will focus on the distribution of big descents over permutations avoiding a single length 3 pattern. Along the way, we uncover an interesting joint equidistribution between two pairs of statistics over 231-avoiding permutations---namely, $(\bdes,\des)$ and $(\pk,\des)$ where $\pk$ is the peak number statistic---as well as a surprising new interpretation for the Narayana numbers in terms of ``right big descents'' of 123-avoiding permutations. Section \ref{s-double} will be devoted to our results for permutations avoiding two patterns of length 3. We suggest some future directions of research in Section \ref{s-future}, including several conjectures concerning real-rootedness, log-concavity, and Schur positivity. The proofs of two technical lemmas needed for our study of 123-avoiding permutations are delayed until Appendix \ref{s-B}.

\section{\label{s-single}Results for single patterns}

In this section, we study the distribution of $\bdes$ over permutations avoiding a single pattern of length 3. By reverse-complementation symmetry, we only need to consider the patterns $123$, $132$, $231$, and $321$, and it is readily checked that they fall into different $\bdes$-Wilf equivalence classes, thus confirming Theorem \ref{t-bdesW1}. Moreover, we derive a formula for the distribution of $\bdes$ over each of these $\bdes$-Wilf equivalence classes.

\subsection{The pattern 132}

Let us begin by stating the following formula for the generating function $B(t,x;132)$.

\begin{thm}
\label{t-132gf}We have 
\begin{multline*}
B(t,x;132)\\
\quad=\frac{1-2(1-t)x+(1-t)(1-2t)x^{2}-\sqrt{1-4x+6(1-t)x^{2}-4(1-t)^{2}x^{3}+(1-t)^{2}x^{4}}}{2tx(1-(1-t)x)}.
\end{multline*}
\end{thm}

Table \ref{tb-132} displays the first ten polynomials $B_{n}(t;132)$, obtained by extracting the coefficient of $x^{n}$ in $B(t,x;132)$ for $0\leq n\leq9$.

\renewcommand{\arraystretch}{1.2}
\begin{table}
\begin{centering}
\begin{tabular}{|c|c|c|c|c|}
\cline{1-2} \cline{2-2} \cline{4-5} \cline{5-5} 
$n$ & $B_{n}(t;132)$ & \quad{} & $n$ & $B_{n}(t;132)$\tabularnewline
\cline{1-2} \cline{2-2} \cline{4-5} \cline{5-5} 
$0$ & $1$ &  & $5$ & $5+25t+12t^{2}$\tabularnewline
\cline{1-2} \cline{2-2} \cline{4-5} \cline{5-5} 
$1$ & $1$ &  & $6$ & $6+55t+64t^{2}+7t^{3}$\tabularnewline
\cline{1-2} \cline{2-2} \cline{4-5} \cline{5-5} 
$2$ & $2$ &  & $7$ & $7+105t+233t^{2}+82t^{3}+2t^{4}$\tabularnewline
\cline{1-2} \cline{2-2} \cline{4-5} \cline{5-5} 
$3$ & $3+2t$ &  & $8$ & $8+182t+674t^{2}+505t^{3}+61t^{4}$\tabularnewline
\cline{1-2} \cline{2-2} \cline{4-5} \cline{5-5} 
$4$ & $4+9t+t^{2}$ &  & $9$ & $9+294t+1668t^{2}+2206t^{3}+660t^{4}+25t^{5}$\tabularnewline
\cline{1-2} \cline{2-2} \cline{4-5} \cline{5-5} 
\end{tabular}
\par\end{centering}
\caption{\label{tb-132}Distribution of $\bdes$ over $\mathfrak{S}_{n}(132)$ for $0\leq n\leq9$.}
\end{table}

For our proof of Theorem \ref{t-132gf}, we shall need the following system of functional equations for the generating functions $B(t,x;132)$ and 
\[
\bar{B}(t,x;132)\coloneqq\sum_{n=1}^{\infty}\sum_{\substack{\pi\in\mathfrak{S}_{n}(132)\\
\pi_{1}=n}}t^{\bdes(\pi)}x^{n}.
\]

\begin{lem}
\label{l-132funeq}Abbreviate $B\coloneqq B(t,x;132)$ and $\bar{B}\coloneqq\bar{B}(t,x;132)$. We have
\begin{align}
B &= 1+\bar{B}+x(B-1)+tx(B-1)^{2}\quad\text{and}\label{e-132feq1}\\
\bar{B} &= x(1+\bar{B}+t(B-\bar{B}-1)).\label{e-132feq2}
\end{align}
\end{lem}

\begin{proof}
Consider $\pi\in\mathfrak{S}_{n}(132)$ with $\pi_{1}=n$, and write $\pi=n\tau$. The case when $\tau$ is empty contributes the term $x$ to \eqref{e-132feq2}. Otherwise, suppose that $\tau$ begins with $n-1$; then $\bdes(\pi)=\bdes(\tau)$, so this case contributes $x\bar{B}$ to \eqref{e-132feq2}. Finally, if $\tau$ does not begin with $n-1$ (and is nonempty), then $\bdes(\pi)=\bdes(\tau)+1$, so this case contributes $tx(B-\bar{B}-1)$ to \eqref{e-132feq2}. Thus, \eqref{e-132feq2} is proven.

For \eqref{e-132feq1}, the empty permutation contributes the term $1$, and every nonempty $\pi\in\mathfrak{S}_{n}(132)$ can be uniquely decomposed as $\pi=\sigma n\tau$ where both $\sigma$ and $\tau$ are 132-avoiding and every letter in $\tau$ is smaller than every letter in $\sigma$. The case when $\sigma$ is empty contributes $\bar{B}$ to \eqref{e-132feq1}. If $\sigma$ is nonempty but $\tau$ is empty, then $\pi=\sigma n$ and thus $\bdes(\pi)=\bdes(\sigma)$, so this case contributes $x(B-1)$ to \eqref{e-132feq1}. Finally, suppose that neither $\sigma$ nor $\tau$ is empty. Note that $\tau$ cannot begin with $n-1$, since every letter in $\tau$ is smaller than every letter in $\sigma$, so we have $\bdes(\pi)=\bdes(\sigma)+\bdes(\tau)+1$. Therefore, this last case contributes $tx(B-1)^{2}$ to \eqref{e-132feq1}, and we are done.
\end{proof}

Let us now prove Theorem \ref{t-132gf}.

\begin{proof}[Proof of Theorem \ref{t-132gf}]
Solving \eqref{e-132feq2} for $\bar{B}$ yields
\[
\bar{B}=\frac{x\left(1+t(B-1)\right)}{1-(1-t)x}.
\]
Substituting this expression into \eqref{e-132feq1} and solving for $B$ gives the stated formula.
\end{proof}

\subsection{Interlude: Dyck paths}

Before continuing, let us briefly review some basic definitions and notions surrounding Dyck paths, which we will need for studying big descents in 231-, 321-, and 123-avoiding permutations.

A \textit{Dyck path} of semilength $n$ is a lattice path in $\mathbb{Z}^{2}$\textemdash consisting of $2n$ steps from the step set $\{(1,1),(1,-1)\}$\textemdash which starts at the origin $(0,0)$, ends on the $x$-axis, and never traverses below the $x$-axis. Let $\mathcal{D}_{n}$ denote the set of Dyck paths of semilength $n$, and let $\varepsilon$ denote the empty path (the only Dyck path of semilength 0). The steps $(1,1)$ are called \textit{up steps} and denoted by the letter $U$, and the $(1,-1)$ are called \textit{down steps} and denoted by $D$. Thus, we can represent Dyck paths as words on the alphabet $\{U,D\}$. See Figure \ref{f-Dyck} for an example.

When we refer to a factor of a Dyck path, we really mean a factor (i.e., a consecutive subsequence) of the corresponding word. Given a Dyck path $\mu$ and a word $\alpha$ on the alphabet $\{U,D\}$, let $\occ_{\alpha}(\mu)$ denote the number of $\alpha$-factors in $\mu$\textemdash i.e., occurrences of $\alpha$ as a factor in $\mu$. For example, the Dyck path $\mu$ in Figure \ref{f-Dyck} contains four $UD$-factors, so $\occ_{UD}(\mu)=4$. We will also use the notation $\occ_{\alpha}^{0}(\mu)$ for the number of $\alpha$-factors in $\mu$ that start at level 0 (i.e., the occurrence of the factor starts at the $x$-axis). Continuing the example in Figure \ref{f-Dyck}, we see that only the third of the four $UD$-factors begins at level 0, so $\occ_{UD}^{0}(\mu)=1$.

\begin{figure}
\begin{center}
\begin{tikzpicture}[scale=0.5] 
\draw [line width=0] (4,2); 
\draw[pathcolorlight] (0,0) -- (14,0); 
\drawlinedots{0,1,2,3,4,5,6,7,8,9,10,11,12,13,14}{0,1,2,3,2,1,2,1,0,1,0,1,2,1,0} 
\end{tikzpicture}
\end{center}
\vspace{-10bp}
\caption{\label{f-Dyck}The Dyck path $\mu\in\mathcal{D}_{7}$ corresponding to the word $UUUDDUDDUDUUDD$.}
\end{figure}
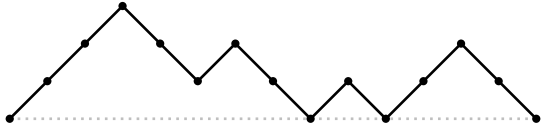

Let us describe two recursive decompositions for Dyck paths that will play an important role going forward. First, consider the \textit{first return decomposition} of Dyck paths: every Dyck path $\mu$ is either empty or can be uniquely decomposed in the form 
\[
\mu=U\mu_{1}D\mu_{2}
\]
where $\mu_{1}$ and $\mu_{2}$ are Dyck paths. Similarly, we have the \textit{last return decomposition}: every Dyck path $\mu$ is either empty or can be uniquely decomposed in the form 
\[
\mu=\mu_{1}U\mu_{2}D
\]
where $\mu_{1}$ and $\mu_{2}$ are Dyck paths. See Figure \ref{f-Dyck-fr} for the first and last return decompositions of our running example.

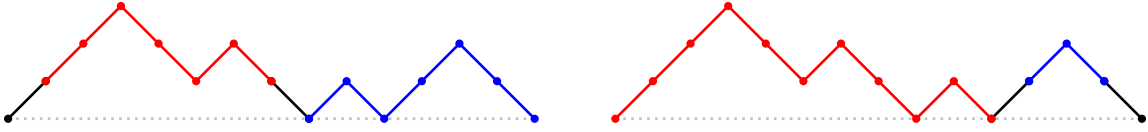
\begin{figure}
\begin{center}
\begin{tikzpicture}[scale=0.5] 
\draw [line width=0] (4,2); 
\draw[pathcolorlight] (0,0) -- (14,0); 
\drawlinedots{0,1}{0,1};
\drawlinedots{7,8}{1,0};
\drawlinedotsred{1,2,3,4,5,6,7}{1,2,3,2,1,2,1}
\drawlinedotsblue{8,9,10,11,12,13,14}{0,1,0,1,2,1,0};
\end{tikzpicture}
\qquad
\begin{tikzpicture}[scale=0.5] 
\draw [line width=0] (4,2); 
\draw[pathcolorlight] (0,0) -- (14,0);
\drawlinedots{10,11}{0,1};
\drawlinedots{13,14}{1,0};
\drawlinedotsred{0,1,2,3,4,5,6,7,8,9,10}{0,1,2,3,2,1,2,1,0,1,0}
\drawlinedotsblue{11,12,13}{1,2,1};
\end{tikzpicture}
\end{center}
\vspace{-10bp}
\caption{\label{f-Dyck-fr}The first and last return decompositions of the Dyck path $\mu$ in Figure \ref{f-Dyck}, where the paths $\mu_{1}$ are colored \textcolor{red}{red} and the $\mu_{2}$ colored \textcolor{blue}{blue}.}
\end{figure}

Finally, we note that Dyck paths are sometimes viewed as paths from $(0,0)$ to $(n,n)$, with unit north and east steps---playing the roles of $U$ and $D$, respectively---that never traverse below the diagonal line $y=x$. It will be convenient to take this view in Sections~\ref{ss-321} and~\ref{ss-123}.

\subsection{\label{ss-231}The pattern 231}

The main result of this section is a joint equidistribution over 231-avoiding permutations, from which a formula for counting 231-avoiding permutations by big descents will follow. Let us first define several additional permutation statistics which will play a role in this section.
\begin{itemize}
\item Call $k\in[n-1]$ a \textit{small descent} of $\pi\in\mathfrak{S}_{n}$ if $\pi_{k}=\pi_{k+1}+1$\textemdash i.e., if $k$ is a descent that is not a big descent.\footnote{Small descents are also called \textit{reverse successions} or \textit{falling successions} in the literature.} Let $\sdes(\pi)$ denote the number of small descents of $\pi$.
\item Call $k\in[n-1]$ a \textit{left double descent} of $\pi\in\mathfrak{S}_{n}$ if either $k=1$ and $\pi_{1}>\pi_{2}$, or if $\pi_{k-1}>\pi_{k}>\pi_{k+1}$. Let $\lddes(\pi)$ denote the number of left double descents of $\pi$.
\item Call $k\in\{2,3,\dots,n-1\}$ a \textit{peak} of $\pi\in\mathfrak{S}_{n}$ if $\pi_{k-1}<\pi_{k}>\pi_{k+1}$. Let $\pk(\pi)$ denote the number of peaks of $\pi$.
\end{itemize}
For example, let $\pi=214863975$. Then $\sdes(\pi)=1$, $\lddes(\pi)=3$, and $\pk(\pi)=2$.

\begin{thm} \label{t-231joint}
For all $n\geq0$, the statistics $(\bdes,\des)$ and $(\pk,\des)$ are jointly equidistributed over $\mathfrak{S}_{n}(231)$.
\end{thm}

Let us discuss a few consequences of Theorem \ref{t-231joint} before proving it.

\begin{cor}
\label{c-231jointform} For all $n,j,k\geq0$, the number of permutations in $\mathfrak{S}_{n}(231)$ with $\bdes(\pi)=k$ and $\des(\pi)=j$ is 
\begin{equation}
\frac{1}{k+1}{2k \choose k}{n-1 \choose 2k}{n-2k-1 \choose j-k}.\label{e-231bdesdes}
\end{equation}
\end{cor}

\begin{proof}
It is known \cite[Theorem 5.6]{Zhuang2017} that \eqref{e-231bdesdes} is the number of permutations in $\mathfrak{S}_{n}(231)$ with $k$ peaks and $j$ descents. The result then follows from the equidistribution given in Theorem \ref{t-231joint}.
\end{proof}

\begin{cor} \label{c-231} 
For all $n,k\geq0$, we have
\begin{equation*}
b_{n,k}(231)=\frac{2^{n-2k-1}}{k+1}{n-1 \choose 2k}{2k \choose k}.\label{e-231bdes}
\end{equation*}
\end{cor}

\begin{proof}
Observe that
\[
\sum_{j=0}^{n-1}{n-2k-1 \choose j-k}=\sum_{j=k}^{n-k-1}{n-2k-1 \choose j-k}=\sum_{j=0}^{n-2k-1}{n-2k-1 \choose j}=2^{k-2k-1}.
\]
This equality, along with Corollary \ref{c-231jointform}, implies 
\begin{align*}
b_{n,k}(231)=\sum_{j=0}^{n-1}\frac{1}{k+1}{2k \choose k}{n-1 \choose 2k}{n-2k-1 \choose j-k} & =\frac{2^{n-2k-1}}{k+1}{n-1 \choose 2k}{2k \choose k}.\qedhere
\end{align*}
\end{proof}

We list the first ten polynomials $B_{n}(t;231)$ in Table \ref{tb-231}.

\begin{table}
\begin{centering}
\begin{tabular}{|c|c|c|c|c|}
\cline{1-2} \cline{2-2} \cline{4-5} \cline{5-5} 
$n$ & $B_{n}(t;231)$ & \quad{} & $n$ & $B_{n}(t;231)$\tabularnewline
\cline{1-2} \cline{2-2} \cline{4-5} \cline{5-5} 
$0$ & $1$ &  & $5$ & $16+24t+2t^{2}$\tabularnewline
\cline{1-2} \cline{2-2} \cline{4-5} \cline{5-5} 
$1$ & $1$ &  & $6$ & $32+80t+20t^{2}$\tabularnewline
\cline{1-2} \cline{2-2} \cline{4-5} \cline{5-5} 
$2$ & $2$ &  & $7$ & $64+240t+120t^{2}+5t^{3}$\tabularnewline
\cline{1-2} \cline{2-2} \cline{4-5} \cline{5-5} 
$3$ & $4+t$ &  & $8$ & $128+672t+560t^{2}+70t^{3}$\tabularnewline
\cline{1-2} \cline{2-2} \cline{4-5} \cline{5-5} 
$4$ & $8+6t$ &  & $9$ & $256+1792t+2240t^{2}+560t^{3}+14t^{2}$\tabularnewline
\cline{1-2} \cline{2-2} \cline{4-5} \cline{5-5} 
\end{tabular}
\par\end{centering}
\caption{\label{tb-231}Distribution of $\bdes$ over $\mathfrak{S}_{n}(231)$
for $0\leq n\leq9$.}
\end{table}

Our proof of Theorem \ref{t-231joint} will be bijective, involving two bijections from 231-avoiding permutations to Dyck paths.
To set up these bijections, we shall need the following recursive decomposition for 231-avoiding permutations. If $\pi\in\mathfrak{S}_{n}(231)$ and $n\geq1$, then we can write $\pi=\sigma n\tau$, so that $\sigma$ is the sequence of letters appearing before $n$ and $\tau$ is the sequence of letters appearing after $n$. Then both $\sigma$ and $\std(\tau)$ are 231-avoiding permutations, and the letters of $\tau$ are greater than those of $\sigma$. Conversely, if we are given 231-avoiding permutations $\sigma$ and $\bar\tau$ of total length $n-1$, then $\sigma n\tau\in\mathfrak{S}_{n}(231)$ where $\tau$ is the sequence obtained by adding $\left|\sigma\right|$ to each letter in $\bar{\tau}$.

Now, we use the above recursive decomposition for 231-avoiding permutations along with the first and last return decompositions for Dyck paths to construct the following recursive bijections $\omega_{f},\omega_{\ell}\colon\mathfrak{S}_{n}(231)\rightarrow\mathcal{D}_{n}$.
If $\pi$ is empty, let $\omega_{f}(\pi)=\omega_{\ell}(\pi)=\varepsilon$; otherwise, write $\pi=\sigma n\tau$ and let
\begin{align*} 
\omega_{f}(\pi) \coloneqq U\omega_{f}(\sigma)D\omega_{f}(\std(\tau))\qquad \text{and}\qquad
\omega_{\ell}(\pi) \coloneqq \omega_{\ell}(\std(\tau))U\omega_{\ell}(\sigma)D.
\end{align*}

\noindent\begin{minipage}{\textwidth}
\begin{lem} \label{l-231stats}
For all $n\geq0$ and $\pi\in\mathfrak{S}_{n}(231)$, we have\textup{:}
\begin{multicols}{2}
\begin{enumerate}
\item [\normalfont{(a)}] $\bdes(\pi)=\occ_{DUU}(\omega_{f}(\pi))$,
\item [\normalfont{(b)}] $\pk(\pi)=\occ_{DUU}(\omega_{\ell}(\pi))$,
\item [\normalfont{(c)}] $\des(\pi)=\occ_{DU}(\omega_{f}(\pi))$, and
\item [\normalfont{(d)}] $\des(\pi)=\occ_{DU}(\omega_{\ell}(\pi))$.
\end{enumerate}
\end{multicols}
\end{lem}
\end{minipage}

\begin{proof}
We only prove (a); the proofs of (b)--(d) are similar.

We proceed by induction on the length $n$ of $\pi$, with the base case $(n=0)$ being trivial. Assume that $\bdes(\pi)=\occ_{DUU}(\omega_{f}(\pi))$ for all 231-avoiding permutations of length up to a fixed $m\geq0$. Let $\pi\in\mathfrak{S}_{m+1}$, and write $\pi=\sigma\:(m+1)\:\tau$. For convenience, let $\bar{\tau}=\std(\tau)$. Since $\sigma$ and $\tau$ each have length at most $m$, we have $\bdes(\sigma)=\occ_{DUU}(\omega_{f}(\sigma))$ and $\bdes(\bar{\tau})=\occ_{DUU}(\omega_{f}(\bar{\tau}))$. Observe that 
\[
\bdes(\pi)=\begin{cases}
\bdes(\sigma)+\bdes(\bar{\tau}), & \text{if }\tau\text{ begins with }m,\\
\bdes(\sigma)+\bdes(\bar{\tau})+1, & \text{otherwise,}
\end{cases}
\]
and
\[
\occ_{DUU}(\omega_{f}(\pi))=\begin{cases}
\occ_{DUU}(\omega_{f}(\sigma))+\occ_{DUU}(\omega_{f}(\bar{\tau}))+1, & \text{if }\omega_{f}(\bar{\tau})\text{ begins with }UU,\\
\occ_{DUU}(\omega_{f}(\sigma))+\occ_{DUU}(\omega_{f}(\bar{\tau})), & \text{otherwise.}
\end{cases}
\]
Let $k$ be the length of $\tau$, and write $\bar{\tau}=\hat{\sigma}k\hat{\tau}$. If $\tau$ begins with $m$, then $\bar{\tau}$ begins with $k$, so $\omega_{f}(\bar{\tau})$ begins with $UD$, and we have 
\[
\bdes(\pi)=\bdes(\sigma)+\bdes(\bar{\tau})=\occ_{DUU}(\omega_{f}(\sigma))+\occ_{DUU}(\omega_{f}(\bar{\tau}))=\occ_{DUU}(\omega_{f}(\pi))
\]
as desired. Now suppose that $\tau$ does not begin with $m$, so that $\hat{\sigma}$ is nonempty. Then $\omega_{f}(\hat{\sigma})$ is nonempty and thus begins with $U$, which implies $\omega_{f}(\bar{\tau})=U\omega_{f}(\hat{\sigma})D\omega_{f}(\hat{\tau})$ begins with $UU$. Therefore, we have
\begin{align*}
\bdes(\pi) & =\bdes(\sigma)+\bdes(\bar{\tau})+1=\occ_{DUU}(\omega_{f}(\sigma))+\occ_{DUU}(\omega_{f}(\bar{\tau}))+1=\occ_{DUU}(\omega_{f}(\pi)),
\end{align*}
completing the proof.
\end{proof}

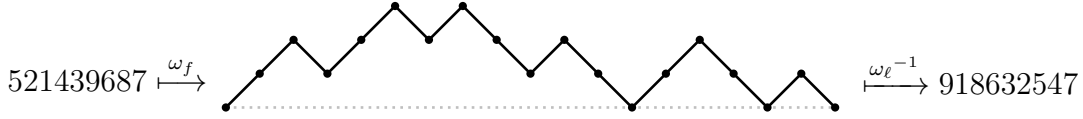
\begin{figure}
\begin{center}
\begin{tikzpicture}[scale=0.45] 
\node (1) at (-3.5,1) {$521439687 \overset{\omega_f}{\longmapsto}$};
\draw[pathcolorlight] (0,0) -- (18,0); 
\drawlinedots{0,1,2,3,4,5,6,7,8,9,10,11,12,13,14,15,16,17,18}{0,1,2,1,2,3,2,3,2,1,2,1,0,1,2,1,0,1,0}
\node (2) at (22,1) {$\xmapsto{{\omega_\ell}^{-1}} 918632547$};
\end{tikzpicture}
\end{center}
\vspace{-10bp}\caption{\label{f-231bij}An example illustrating the bijection $\omega_{\ell}^{-1}\circ\omega_{f}$.}
\end{figure}

It follows from Lemma \ref{l-231stats} that $\omega_{\ell}^{-1}\circ\omega_{f}$ is a bijection from $\mathfrak{S}_{n}(231)$ to itself which sends the number of big descents to the number of peaks, while preserving the number of descents, thus proving Theorem \ref{t-231joint}. We provide an example illustrating this bijection in Figure~\ref{f-231bij}.

Before continuing, let us give one more corollary of Theorem \ref{t-231joint}.
\begin{cor}
For all $n\geq0$, the statistics $\sdes$ and $\lddes$ are equidistributed over $\mathfrak{S}_{n}(231)$.
\end{cor}

\begin{proof}
Since $\des(\pi)=\bdes(\pi)+\sdes(\pi)$ and $\des(\pi)=\pk(\pi)+\lddes(\pi)$ for all $\pi$, Theorem~\ref{t-231joint} implies the desired result.
\end{proof}
The distribution of $\sdes$ (equivalently, of $\lddes$) over $\mathfrak{S}_{n}(231)$ is given by entry A091869 of the OEIS \cite{oeis}. This sequence counts Dyck paths by the number of $DUD$-factors, which is consistent with Lemma \ref{l-231stats}. Indeed, since $\omega_{f}$ sends descents to $DU$-factors and big descents to $DUU$-factors, small descents correspond to $DU$-factors which are not part of a $DUU$-factor, and so they must be part of a $DUD$-factor.

\subsection{\label{ss-321}The pattern 321}

The following generating function gives the distribution of $\bdes$ over 321-avoiding permutations.

\begin{thm} \label{t-321gf}
We have 
\[
B(t,x;321)=\frac{2}{1-2(1-t)x^{2}+\sqrt{1-4x+4(1-t)x^{2}}}.
\]
\end{thm}

See Table \ref{tb-321} for the first ten polynomials $B_{n}(t;321)$.

\begin{table}
\begin{centering}
\begin{tabular}{|c|c|c|c|c|}
\cline{1-2} \cline{2-2} \cline{4-5} \cline{5-5} 
$n$ & $B_{n}(t;321)$ & \quad{} & $n$ & $B_{n}(t;321)$\tabularnewline
\cline{1-2} \cline{2-2} \cline{4-5} \cline{5-5} 
$0$ & $1$ &  & $5$ & $8+26t+8t^{2}$\tabularnewline
\cline{1-2} \cline{2-2} \cline{4-5} \cline{5-5} 
$1$ & $1$ &  & $6$ & $13+72t+45t^{2}+2t^{3}$\tabularnewline
\cline{1-2} \cline{2-2} \cline{4-5} \cline{5-5} 
$2$ & $2$ &  & $7$ & $21+184t+196t^{2}+28t^{3}$\tabularnewline
\cline{1-2} \cline{2-2} \cline{4-5} \cline{5-5} 
$3$ & $3+2t$ &  & $8$ & $34+444t+732t^{2}+214t^{3}+6t^{4}$\tabularnewline
\cline{1-2} \cline{2-2} \cline{4-5} \cline{5-5} 
$4$ & $5+8t+t^{2}$ &  & $9$ & $55+1030t+2454t^{2}+1220t^{3}+103t^{4}$\tabularnewline
\cline{1-2} \cline{2-2} \cline{4-5} \cline{5-5} 
\end{tabular}
\par\end{centering}
\caption{\label{tb-321}Distribution of $\bdes$ over $\mathfrak{S}_{n}(321)$
for $0\leq n\leq9$.}
\end{table}

To prove Theorem \ref{t-321gf} (and the analogous result for the pattern $123$ in the next section), we will make use of a bijection $\chi\colon\mathfrak{S}_{n}(321)\rightarrow\mathcal{D}_{n}$ from 321-avoiding permutations to Dyck paths, which has repeatedly appeared in the literature in equivalent forms---see, e.g., \cite{Baril2017, Elizalde2011, Krattenthaler2001}.
To define this bijection, we first note that any permutation $\pi\in\mathfrak{S}_{n}$ can be displayed diagrammatically by plotting each of the points $(k,\pi_{k})$ in an $n\times n$ grid, where the first coordinate indicates the column and the second indicates the row, with columns numbered $1$ through $n$ from left to right, and 
rows are numbered $1$ through $n$ from bottom to top. The resulting diagram is called the \textit{array} of $\pi$. Then, given $\pi\in\mathfrak{S}_{n}(321)$, define $\chi(\pi)$ to be the Dyck path drawn from the lower-left corner to the upper-right corner of the array of $\pi$, with unit north and east steps, which leaves all of the points $(k,\pi_{k})$ to its right yet remains as close as possible to the main diagonal. See Figure \ref{f-kratbij} for an example illustrating the bijection~$\chi$. 

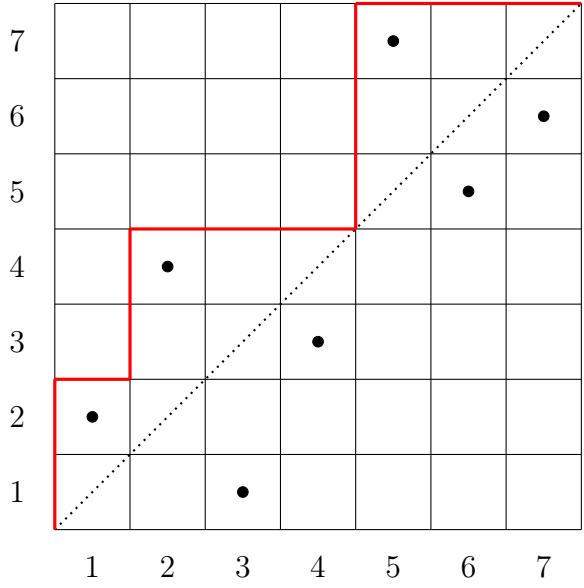
\begin{figure}
\begin{center}
\begin{tikzpicture}         
\foreach \x in {0,1,...,7}             
{                 
\draw (-3, -3 + \x) -- (4, -3 + \x);             
}
\foreach \x in {0,1,...,7}             
{                 
\draw (-3 + \x, -3) -- (-3 + \x, 4);             
}         
\foreach \x in {1,...,7}             
{                 
\draw (-3.5 + \x, -3.5) node {\x};             
}                      
\foreach \x in {1,...,7}             
{                 
\draw (-3.5, -3.5 + \x) node {\x};             
}
        
\filldraw[black]
(-2.5, -1.5) circle [radius=2pt]                    
(-1.5, 0.5) circle [radius=2pt]                    
(-.5, -2.5) circle [radius=2pt]                    
(.5, -.5) circle [radius=2pt]                    
(1.5, 3.5) circle [radius=2pt]                    
(2.5, 1.5) circle [radius=2pt]
(3.5, 2.5) circle [radius=2pt];         
\draw[dotted, thick] (-3, -3) -- (4, 4);
        
\draw[red, very thick] (-3, -3) -- (-3, -1);         
\draw[red, very thick] (-3, -1) -- (-2, -1);         
\draw[red, very thick] (-2, -1) -- (-2, 1);         
\draw[red, very thick] (-2, 1) -- (1, 1);         
\draw[red, very thick] (1, 1) -- (1, 4);                  
\draw[red, very thick] (1, 4) -- (4, 4);              
\end{tikzpicture}
\end{center}
\vspace{-15bp}\caption{\label{f-kratbij}The Dyck path $\chi(2413756)$.}
\end{figure}

Our strategy for proving Theorem \ref{t-321gf} is as follows. First, recall from Section \ref{ss-231} that a descent $k\in[n-1]$ of $\pi\in\mathfrak{S}_{n}$ is a \textit{small descent} of $\pi$ if it is not a big descent of $\pi$, and that $\sdes(\pi)$ denotes the number of small descents of $\pi$. We shall show that the descents of a 321-avoiding permutation $\pi$ correspond to $UDD$-factors of the Dyck path $\chi(\pi)$, and that the small descents of $\pi$ correspond to $UUDD$-factors of $\chi(\pi)$ that begin at level 0. We will then derive a generating function for the joint distribution of $\occ_{UDD}$ and $\occ_{UUDD}^{0}$ over Dyck paths, and specializing this generating function appropriately will prove Theorem \ref{t-321gf}.

\begin{lem} \label{l-321sdes}
Suppose $k\in[n-1]$ is a small descent of $\pi\in\mathfrak{S}_{n}(321)$. Then $\pi_{1}\pi_{2}\dots\pi_{k-1}\in\mathfrak{S}_{k-1}(321)$, $\pi_{k}=k+1$, and $\pi_{k+1}=k$.
\end{lem}

\begin{proof}
First, notice that none of the letters appearing before $\pi_{k}$ in $\pi$ can be larger than $\pi_{k}$; otherwise, that letter along with $\pi_{k}$ and $\pi_{k+1}$ would be an occurrence of 321. Moreover, all of the letters $1,2,\dots,\pi_{k}-2$ must appear before $\pi_{k}$ in $\pi$. Indeed, if one of these letters appears after $\pi_{k}$, then it must appear after $\pi_{k+1}$, giving us an occurrence of 321. Therefore, the prefix $\pi_{1}\pi_{2}\dots\pi_{k-1}$ of $\pi$ consists precisely of the letters $1,2,\dots,\pi_{k}-2$, which implies $k-1=\pi_{k}-2$ and thus $\pi_{k}=k+1$ and $\pi_{k+1}=k$.
\end{proof}

\begin{lem} \label{l-321dyck}
For all $n\geq0$ and $\pi\in\mathfrak{S}_{n}(321)$, we have\textup{:}
\begin{enumerate}
\item [\normalfont{(a)}] $\des(\pi)=\occ_{UDD}(\chi(\pi))$ and
\item [\normalfont{(b)}] $\sdes(\pi)=\occ_{UUDD}^{0}(\chi(\pi))$.
\end{enumerate}
\end{lem}

\begin{proof}
Part (a) is already known; see the proof of \cite[Lemma 1]{Baril2017}. To prove (b), it suffices to show that $k\in[n-1]$ is a small descent of $\pi\in\mathfrak{S}_{n}$ if and only if $\chi(\pi)$ has a $UUDD$-factor at level 0 starting at position $2k-1$, i.e., consisting of the $(2k-1)$th through $(2k+2)$th steps of $\chi(\pi)$.

Let $k$ be a small descent of $\pi$. By Lemma \ref{l-321sdes}, we have $\pi_{1}\pi_{2}\dots\pi_{k-1}\in\mathfrak{S}_{k-1}(321)$, $\pi_{k}=k+1$, and $\pi_{k+1}=k$. Since $\pi_{1}\pi_{2}\dots\pi_{k-1}\in\mathfrak{S}_{k-1}(321)$, the first $2k-2$ steps of $\chi(\pi)$ constitute a Dyck path of semilength $k-1$, i.e., $\chi(\pi)$ returns to the diagonal at that point. Furthermore, the conditions $\pi_{k}=k+1$ and $\pi_{k+1}=k$ imply that $\chi(\pi)$ has a $UUDD$-factor immediately afterwards.

Conversely, suppose that $\chi(\pi)$ has a $UUDD$-factor at level 0 starting at position $2k-1$. Since $\chi(\pi)$ returns to the diagonal immediately prior to that factor, if we narrow our view to the $(k-1)\times(k-1)$ subgrid in the lower-left corner, we know that there is a point $(i,\pi_{i})$ in every column $i\in[k-1]$ of that subgrid\textemdash i.e., the prefix $\pi_{1}\pi_{2}\dots\pi_{k-1}$ of $\pi$ consists precisely of the letters $1,2,\dots,k-1$. Having a $UUD$-factor immediately after that return forces $\pi_{k}=k+1$, and since all of the rows in our grid up to this point are occupied except for row $k$, the following down step forces $\pi_{k+1}=k$. Hence, $k$ is a small descent of $\pi$.
\end{proof}

Let us define
\[
W\coloneqq W(s,t,x)=\sum_{n=0}^{\infty}\sum_{\mu\in\mathcal{D}_{n}}s^{\occ_{UUDD}^{0}(\mu)}t^{\occ_{UDD}(\mu)}x^{n}
\]
to be the ordinary generating function giving the joint distribution of $\occ_{UUDD}^{0}$ and $\occ_{UDD}$ over Dyck paths. To derive a formula for $W$, we will need two auxiliary generating functions. Let $\hat{\mathcal{D}}_{n}$ denote the set of \textit{irreducible} Dyck paths of semilength $n$, i.e., Dyck paths with only one return. Define
\begin{align*}
\hat{W} & \coloneqq\hat{W}(s,t,x)=\sum_{n=1}^{\infty}\sum_{\mu\in\mathcal{\hat{D}}_{n}}s^{\occ_{UUDD}^{0}(\mu)}t^{\occ_{UDD}(\mu)}x^{n},\\
V & \coloneqq W(1,t,x)=\sum_{n=0}^{\infty}\sum_{\mu\in\mathcal{D}_{n}}t^{\occ_{UDD}(\mu)}x^{n}.
\end{align*}
It is known \cite{Sapounakis2006} that $V$ satisfies the functional equation $x((t-1)x+1)V^{2}-V+1=0$, which upon solving yields the formula 
\begin{equation}
V=\frac{1-\sqrt{1-4x+4(1-t)x^{2}}}{2x(1-(1-t)x)}.\label{e-321V}
\end{equation}

\begin{lem} \label{l-Whatgf}
We have
\begin{align*}
\hat{W} & =\frac{1-\sqrt{1-4x+4(1-t)x^{2}}}{2}+t(s-1)x^{2}.
\end{align*}
\end{lem}

\begin{proof}
First, we prove the equation 
\begin{equation}
\hat{W}=x+stx^{2}+tx^{2}(V-1)+x(V-1-xV).\label{e-feq321}
\end{equation}
Let $\mu$ be an irreducible path, so that $\mu=U\mu_{1}D$ for some Dyck path $\mu_{1}$. If $\mu_{1}$ is empty, then $\mu=UD$, which contributes $x$ to \eqref{e-feq321}. Similarly, if $\mu_{1}=UD$, then $\mu=UUDD$, which contributes $stx^{2}$ to \eqref{e-feq321}. Otherwise, if $\mu_{1}$ has semilength greater than 1, then $\mu$ has no $UUDD$-factors at level 0, and we split the analysis into two subcases:
\begin{itemize}
\item Suppose $\mu_{1}$ ends with $UD$. Then let us write $\mu_{1}=\mu_{2}UD$ where $\mu_{2}$ is a nonempty Dyck path, so that $\mu=U\mu_{2}UDD$. Paths of this type contribute $tx^{2}(V-1)$ to \eqref{e-feq321}, as $V-1$ is the generating function for nonempty Dyck paths by $UDD$-factors, and we multiply by $tx^{2}$ to account for the additional $UDD$-factor and the difference in semilength.
\item Suppose $\mu_{1}$ does not end with $UD$. Then $\mu$ and $\mu_{1}$ have the same number of $UDD$-factors, and $\mu$ has semilength one greater than $\mu_{1}$. Note that $V-1-xV$ is the generating function for nonempty Dyck paths that do not end with a $UD$, by $UDD$-factors. Multiplying by $x$ accounts for the difference in semilength, and so we have the final $x(V-1-xV)$ term in \eqref{e-feq321}.
\end{itemize}
Hence, \eqref{e-feq321} is proven. The desired formula for $\hat{W}$ is obtained from \eqref{e-feq321} by substituting in \eqref{e-321V} and simplifying.
\end{proof}

\begin{prop} \label{p-Wgf}
We have 
\[
W=\frac{2}{1+2(1-s)tx^{2}+\sqrt{1-4x+4(1-t)x^{2}}}.
\]
\end{prop}

\begin{proof}
The first return decomposition gives $W=1+\hat{W}W$, hence $W=(1-\hat{W})^{-1}$. Substituting in the formula for $\hat{W}$ given by Lemma \ref{l-Whatgf} and simplifying yields the desired formula for $W$.
\end{proof}

We are now ready to prove Theorem \ref{t-321gf}.

\begin{proof}[Proof of Theorem \ref{t-321gf}]
By Lemma \ref{l-321dyck}, we have 
\[
W=\sum_{n=0}^{\infty}\sum_{\pi\in\mathfrak{S}_{n}(321)}s^{\sdes(\pi)}t^{\des(\pi)}x^{n}.
\]
Since $\bdes(\pi)=\des(\pi)-\sdes(\pi)$ for all permutations $\pi$, the formula
\[
B(t,x;321)=W(t^{-1},t,x)=\frac{2}{1-2(1-t)x^{2}+\sqrt{1-4x+4(1-t)x^{2}}}
\]
follows from Proposition \ref{p-Wgf}.
\end{proof}

\subsection{\label{ss-123}The pattern 123}

Lastly, we have the following formula for the pattern $123$.

\begin{thm} \label{t-123}
For all $n,k\geq0$, we have 
\[
b_{n,k}(123)=\frac{2}{n+1}{n+1 \choose k+2}{n-2 \choose k}.
\]
\end{thm}

The first ten polynomials $B_n(t; 123)$ are shown in Table \ref{tb-123}.

\begin{table}
\begin{centering}
\begin{tabular}{|c|c|c|c|c|}
\cline{1-2} \cline{2-2} \cline{4-5} \cline{5-5} 
$n$ & $B_{n}(t;123)$ & \quad{} & $n$ & $B_{n}(t;123)$\tabularnewline
\cline{1-2} \cline{2-2} \cline{4-5} \cline{5-5} 
$0$ & $1$ &  & $5$ & $5+20t+15t^{2}+2t^{3}$\tabularnewline
\cline{1-2} \cline{2-2} \cline{4-5} \cline{5-5} 
$1$ & $1$ &  & $6$ & $6+40t+60t^{2}+24t^{3}+2t^{4}$\tabularnewline
\cline{1-2} \cline{2-2} \cline{4-5} \cline{5-5} 
$2$ & $2$ &  & $7$ & $7+70t+175t^{2}+140t^{3}+35t^{4}+2t^{5}$\tabularnewline
\cline{1-2} \cline{2-2} \cline{4-5} \cline{5-5} 
$3$ & $3+2t$ &  & $8$ & $8+112t+420t^{2}+560t^{3}+280t^{4}+48t^{5}+2t^{6}$\tabularnewline
\cline{1-2} \cline{2-2} \cline{4-5} \cline{5-5} 
$4$ & $4+8t+2t^{2}$ &  & $9$ & $9+168t+882t^{2}+1764t^{3}+1470t^{4}+504t^{5}+63t^{6}+2t^{7}$\tabularnewline
\cline{1-2} \cline{2-2} \cline{4-5} \cline{5-5} 
\end{tabular}
\par\end{centering}
\caption{\label{tb-123}Distribution of $\bdes$ over $\mathfrak{S}_{n}(123)$ for $0\leq n\leq9$.}
\end{table}

As we shall see, our proof of Theorem \ref{t-123} is far more intricate than those of our previous results. Our approach, however, will allow us to simultaneously prove a related result, which features a surprising appearance of the Narayana numbers. Given a permutation $\pi\in\mathfrak{S}_{n}$, we say that $k\in[n]$ is a \textit{right big descent} of $\pi$ if either $k$ is a big descent of $\pi$, or if $k=n$ and $\pi_k>1$. (In other words, right big descents of $\pi$ are big descents of the word $\pi0$ obtained by appending $0$ to $\pi$.) Let $\rbdes(\pi)$ denote the number of right big descents of $\pi$.

\begin{thm} \label{t-123-nar}
For all $n,k\geq0$, the number of permutations in $\mathfrak{S}_{n}(123)$ with $\rbdes(\pi)=k$ is the Narayana number
\[
N_{n,k}\coloneqq\frac{1}{k+1}{n-1 \choose k}{n \choose k}.
\]
\end{thm}

Given $\pi\in\mathfrak{S}_n$, we say that $k\in[n-1]$ is a {\em big ascent} of $\pi$ if $\pi_k+1<\pi_{k+1}$, and that $k\in\{0,1,\dots,n-1\}$ is a {\em left big ascent} of $\pi$ if either $k$ is a big ascent of $\pi$, or if $k=0$ and $\pi_1>1$. By taking reversals, we see that the distribution of (right) big descents over $\mathfrak{S}_{n}(123)$ is the same as that of (left) big ascents over $\mathfrak{S}_{n}(321)$, and it will be more convenient for us to study the latter.

Recall the bijection $\chi\colon \mathfrak{S}_n(321) \rightarrow \mathcal{D}_n$ from Section~\ref{ss-321}. Using this bijection, we seek to translate big ascents of $321$-avoiding permutations into Dyck path statistics. For convenience, we call $UD$-factors of a Dyck path {\em peaks} (not to be confused with peaks of a permutation).

We say that $k\in[n]$ is a {\em weak excedance} of $\pi\in\mathfrak{S}_n$ if $\pi_k\ge k$. It is well known that a permutation $\pi\in\mathfrak{S}_n$ is 321-avoiding if and only if the subsequence of letters corresponding to weak excedances in $\pi$ and the subsequence of remaining letters are both increasing. In terms of the array of $\pi$, the weak excedances correspond to the points above the main diagonal, and the remaining letters correspond to points below the main diagonal. 

Now, let $\pi\in\mathfrak{S}_n(321)$ and let $\mu = \chi(\pi)$ be the corresponding Dyck path. To each point $(k,\pi_k)$ of the permutation array, we associate two steps of $\mu$: the $D$ step in column $k$ and the $U$ step in row $\pi_k$. Then $k\in[n]$ is a weak excedance of $\pi\in\mathfrak{S}_n(321)$ if and only if the associated steps form a peak. For example, see the permutation array and the corresponding Dyck path in Figure \ref{f-kratbij}; there, the letters corresponding to weak excedances are $\pi_1 = 2$, $\pi_2 = 4$, and $\pi_5 = 7$, which indeed form an increasing subsequence corresponding to the peaks of $\mu$.

We can describe the inverse of $\chi$ as follows. Given a Dyck path $\mu$ from the lower-left corner to the upper-right corner of an $n\times n$ grid, start by placing a point in the cell inside each peak $UD$---these points are the weak excedances of the permutation. Let $r$ be the number of peaks of $\mu$, and let $i_1<i_2<\dots<i_{n-r}$ and $j_1<j_2<\dots<j_{n-r}$ be the indices of the rows and columns, respectively, where no point has been placed yet. Now, place points in $(i_1,j_1), (i_2,j_2),\dots,(i_{n-r},j_{n-r})$. The resulting points form the array of the permutation $\pi=\chi^{-1}(\mu)$. Note that $i_1,i_2,\dots,i_{n-r}$ are precisely the indices in $[n]$ that are not weak excedances of $\pi$.

To translate the big ascents of $\pi$ in terms of $\mu$, we will refine such ascents into two types. We say that a big ascent $k$ of $\pi$ is {\em high} if $k+1$ is a weak excedance of $\pi$, and that $k$ is {\em low} otherwise. Denote the number of high and low big ascents of $\pi$ by $\hibasc(\pi)$ and $\lobasc(\pi)$, respectively. Continuing the example in Figure \ref{f-kratbij}, the high big ascents of the permutation $\pi=2413756$ are $1$ and $4$, whereas $3$ is the only low big ascent, so $\hibasc(\pi)=2$ and $\lobasc(\pi)=1$.

High big ascents of $\pi$ are easy to read from $\mu$. Indeed, $k\in[n-1]$ is a high big ascent of $\pi$ if and only if the peak of $\mu$ corresponding to the weak excedance $k+1$ is not immediately preceded by another peak. Define $\hibasc(\mu)$ to be the number of peaks of $\mu$, other than the first one, that are not immediately preceded by another peak. Then $\hibasc(\pi)=\hibasc(\mu)$.

It is also easy to read from $\mu$ whether $k=0$ is a left big ascent of $\pi$, that is, if $\pi_1>1$. This happens precisely when the first peak of $\mu$ is preceded by a $U$ (equivalently, when $\mu$ starts with $UU$). Define $\ini_{UU}(\mu)$ by
\[
\ini_{UU}(\mu)\coloneqq\begin{cases}
1, & \text{if }\mu\text{ starts with }UU,\\
0, & \text{otherwise}.
\end{cases}
\]

Low big ascents of $\pi$ are harder to read from $\mu$. To do this, first color the steps of $\mu$ as follows: color in red the $U$ and $D$ steps that belong to peaks, and color the remaining steps in blue. Among the points in the array of $\pi$, the ones corresponding to weak excedances have both associated steps red, and all other points have both associated steps blue. More precisely, if $i_1<i_2<\dots<i_{n-r}$ are the indices in $[n]$ that are not weak excedances of $\pi$, then for each $\ell\in[n-r]$, the point $(i_\ell,\pi_{i_\ell})$ is associated to the $\ell$th blue $D$ step and the $\ell$th blue $U$ step of $\mu$.

By definition, $k\in[n-1]$ is a low big ascent if and only if $\pi_k+1<\pi_{k+1}<k+1$. In particular, neither $k$ nor $k+1$ are weak excedances in this case, so the steps associated to the points $(k,\pi_k)$ and $(k+1,\pi_{k+1})$ are all blue.
Moreover, the $D$ steps associated to these two points must be consecutive, since they are in adjacent columns, and so having $U$ steps in between would cause the second of the $D$ steps to be red. On the other hand, the $U$ steps associated to these two points cannot be consecutive, since they are in rows $\pi_k$ and $\pi_{k+1}$, which are nonadjacent. 

We deduce that the number of low big ascents of $\pi$ is equal to the number of indices $\ell\in[n-r-1]$ such that the 
$\ell$th and $(\ell+1)$st blue $D$ steps of $\mu$ are consecutive, but the $\ell$th and $(\ell+1)$st blue $U$ steps are not. Denote the number of such indices by $\lobasc(\mu)$.

So far, we have established the following lemma.

\begin{lem} \label{l-123dyck}
For all $n\geq0$ and $\pi\in\mathfrak{S}_{n}(123)$, we have\textup{:}
\begin{enumerate}
\item [\normalfont{(a)}] $\bdes(\pi)=(\hibasc+\lobasc)(\chi(\pi^{r}))$
and
\item [\normalfont{(b)}] $\rbdes(\pi)=(\hibasc+\lobasc+\ini_{UU})(\chi(\pi^{r}))$.
\end{enumerate}
\end{lem}

In addition to $B(t,x;123)$, let us define the generating functions 
\begin{align*}
\grave{B}(t,x;123)&\coloneqq \sum_{n=0}^{\infty}\sum_{\pi\in\mathfrak{S}_{n}(123)}t^{\rbdes(\pi)}x^{n}\quad \text{and} \\
F(u,v,w,x)&\coloneqq \sum_{n=0}^\infty \sum_{\mu\in\mathcal{D}_n} u^{\hibasc(\mu)} v^{\lobasc(\mu)} w^{\ini_{UU}(\mu)} x^n.
\end{align*}
By Lemma \ref{l-123dyck}, we have $B(t,x;123)=F(t,t,1,x)$ and $\grave{B}(t,x;123)=F(t,t,t,x)$, so our goal is now to find an expression for $F(u,v,w,x)$. To do this, we shall require another generating function $G(s,t,z)$ for the joint distribution of two different statistics over Dyck paths.

Given $\nu\in\mathcal{D}_m$, let $\pk(\nu)\coloneqq \occ_{UD}(\nu)$ denote the number of peaks of $\nu$. For $i\in[m]$, let $U_i$ and $D_i$ denote the $i$th $U$ step and the $i$th $D$ step of $\nu$, respectively. Define $\con(\nu)$ to be the number of $i\in[m-1]$ such that $D_i$ and $D_{i+1}$ are consecutive, but $U_i$ and $U_{i+1}$ are not. Then let
$$G(s,t,z)\coloneqq\sum_{m=0}^\infty \sum_{\nu\in\mathcal{D}_m} s^{\pk(\nu)} t^{\con(\nu)} z^m.$$

The next two lemmas provide a formula for $G(s,t,z)$ and explain the connection between the generating functions $F(u,v,w,x)$ and $G(s,t,z)$. We will delay their proofs until Appendix~\ref{s-B}.

\begin{lem}\label{lem:G}
    We have $$G(s,t,z)=\frac{1-(1+s-2t)z-\sqrt{1-2(1+s)z+((1+s)^2-4st)z^2}}{2tz}.$$
\end{lem}

\begin{lem}\label{lem:FG}
    We have
    $$F(u,v,w,x)=\frac{1}{u}\left(w+\frac{ux}{1-x}\right)\left(G(s(u,v,x),t(u,v,x),z(u,v,x))-1\right)+\frac{1}{1-x}$$
    where 
    \begin{align*}
    a(u,v,x) &\coloneqq 1+\frac{ux}{1-x}\left(1+v+\frac{ux}{1-x}\right), & s(u,v,x) &\coloneqq \frac{\frac{ux}{1-x}\left(1+\frac{ux}{1-x}\right)}{a(u,v,x)},\\
    t(u,v,x) &\coloneqq \frac{\left(1+\frac{ux}{1-x}\right)\left(v+\frac{ux}{1-x}\right)}{a(u,v,x)}, \quad \text{and} &  z(u,v,z) &\coloneqq a(u,v,x)\,x.
    \end{align*}
\end{lem}

We are now ready to prove Theorems \ref{t-123} and \ref{t-123-nar}.

\begin{proof}[Proof of Theorems \ref{t-123}--\ref{t-123-nar}] An expression for $F(u,v,w,x)$ is obtained by combining Lemmas~\ref{lem:G} and~\ref{lem:FG}, which leads to the specializations
\begin{align*}
B(t,x;123)&=F(t,t,1,x)\\
&=\frac{1 - 2(1-t^2)x+(1-t)^2x^2-(1-(1-t)x)\sqrt{1-2(1+t)x+(1-t)^2x^2} }{2t^2x}
\end{align*}
and
$$\grave{B}(t,x;123)=F(t,t,t,x)=\frac{1-(1-t)x-\sqrt{1-2(1+t)x+(1-t)^2x^2} }{2tx}.$$
The former is the generating function of the numbers $\frac{2}{n+1}{n+1 \choose k+2}{n-2 \choose k}$ \cite[A108838]{oeis}, and the latter is the generating function of the Narayana numbers \cite[A001263]{oeis}.
\end{proof}

\section{\label{s-double}Results for pairs of patterns}

We now turn our attention to permutations avoiding a pair of patterns. By reverse-complementation symmetry, we only need to consider the pairs $\{123,132\}$, $\{123,231\}$, $\{123,321\}$, $\{132,213\}$, $\{132,321\}$, $\{213,231\}$, $\{213,312\}$, $\{231,312\}$, and $\{231,321\}$. For each of these pattern sets $\Pi$, we will derive a formula for the distribution of $\bdes$ over $\mathfrak{S}_{n}(\Pi)$; Theorem \ref{t-bdesW2} follows immediately from these results.

It was shown by Simion and Schmidt \cite{Simion1985} that for several pattern sets $\Pi\subseteq\mathfrak{S}_{3}$ of size 2, the cardinality of $\mathfrak{S}_{n}(\Pi)$ is $2^{n-1}$, which is also the number of binary words of length $n-1$. Many of our proofs in this section will use bijections with binary words. Let us write $\mathcal{W}_{n}$ for the set of binary words of length $n$, and $\mathcal{W}_{n}^{(1)}$ for the subset of those that end with a 1. For example, we have
\[
\mathcal{W}_{3}=\{000,001,010,011,100,101,110,111\}\quad\text{and}\quad\mathcal{W}_{3}^{(1)}=\{001,011,101,111\}.
\]
As with Dyck paths, given binary words $v$ and $w$, we write $\occ_{v}(w)$ for the number of $v$-factors in $w$. For example, we have $\occ_{011}(10110001011)=2$. We will also use the standard notation $2^{S}$ for the power set of a set $S$.

\subsection{The patterns 213 and 231} \label{ss-213-231}

We begin with $\Pi=\{213,231\}$. The distribution of $\bdes$ over $\mathfrak{S}_{n}(213,231)$ is given by the odd-indexed entries in the $n$th row of Pascal's triangle.

\begin{thm} \label{t-213-231}
For all $n\geq1$ and $k\geq0$, we have 
\[
b_{n,k}(213,231)={n \choose 2k+1}.
\]
\end{thm}

The first ten polynomials $B_{n}(t;213,231)$ are given in Table \ref{tb-213-231}.

\begin{table}
\begin{centering}
\begin{tabular}{|c|c|c|c|c|}
\cline{1-2} \cline{2-2} \cline{4-5} \cline{5-5} 
$n$ & $B_{n}(t;213,231)$ & \quad{} & $n$ & $B_{n}(t;213,231)$\tabularnewline
\cline{1-2} \cline{2-2} \cline{4-5} \cline{5-5} 
$0$ & $1$ &  & $5$ & $5+10t+t^{2}$\tabularnewline
\cline{1-2} \cline{2-2} \cline{4-5} \cline{5-5} 
$1$ & $1$ &  & $6$ & $6+20t+6t^{2}$\tabularnewline
\cline{1-2} \cline{2-2} \cline{4-5} \cline{5-5} 
$2$ & $2$ &  & $7$ & $7+35t+21t^{2}+t^{3}$\tabularnewline
\cline{1-2} \cline{2-2} \cline{4-5} \cline{5-5} 
$3$ & $3+t$ &  & $8$ & $8+56t+56t^{2}+8t^{3}$\tabularnewline
\cline{1-2} \cline{2-2} \cline{4-5} \cline{5-5} 
$4$ & $4+4t$ &  & $9$ & $9+84t+126t^{2}+36t^{3}+t^{4}$\tabularnewline
\cline{1-2} \cline{2-2} \cline{4-5} \cline{5-5} 
\end{tabular}
\par\end{centering}
\caption{\label{tb-213-231}Distribution of $\bdes$ over $\mathfrak{S}_{n}(213,231)$ for $0\leq n\leq9$.}
\end{table}

To prove Theorem \ref{t-213-231}, we will make use of a bijection $\phi_{213,231}\colon\mathfrak{S}_{n}(213,231)\rightarrow\mathcal{W}_{n-1}$ described by Bukata et al.\ \cite{Bukata2019}, which is equivalent to an earlier bijection due to Simion and Schmidt \cite{Simion1985} from $\mathfrak{S}_{n}(132,312)$ to $2^{\{2,3,\dots,n\}}$. To set up this bijection, let us describe the structure of permutations $\pi\in\mathfrak{S}_{n}(213,231)$. For each $k\in[n-1]$, either $\pi_{k}=\min\{\pi_{k},\pi_{k+1},\dots,\pi_{n}\}$ or $\pi_{k}=\max\{\pi_{k},\pi_{k+1},\dots,\pi_{n}\}$. Otherwise, $\pi_{k}$ along with the letters $\min\{\pi_{k},\pi_{k+1},\dots,\pi_{n}\}$ and $\max\{\pi_{k},\pi_{k+1},\dots,\pi_{n}\}$ would be either an occurrence of 213 or of 231. In fact, $\pi$ is completely determined by whether, for each $k\in[n-1]$, the letter $\pi_{k}$ is the smallest or the largest among all letters in the suffix $\pi_{k}\pi_{k+1}\dots\pi_{n}$. Thus, we may define the bijection
\[
\phi_{213,231}(\pi)=w_{1}w_{2}\dots w_{n-1}\in\mathcal{W}_{n-1}
\]
where
\[
w_{k}=\begin{cases}
1, & \text{if }\pi_{k}=\min\{\pi_{k},\pi_{k+1},\dots,\pi_{n}\},\\
0, & \text{if }\pi_{k}=\max\{\pi_{k},\pi_{k+1},\dots,\pi_{n}\},
\end{cases}
\]
for each $k\in[n-1]$. For example, we have $\phi_{213,231}(91238476)=01110100$.

\begin{lem}
\label{l-213-231}For all $n\geq1$ and $\pi\in\mathfrak{S}_{n}(213,231)$, we have $\bdes(\pi)=\occ_{01}(\phi_{213,231}(\pi)).$
\end{lem}

\begin{proof}
Let $\pi\in\mathfrak{S}_{n}(213,231)$ and $w=w_{1}w_{2}\dots w_{n-1}=\phi_{213,231}(\pi)$. We prove that $k\in[n-1]$ is a big descent of $\pi$ if and only if $w_{k}=0$ and $w_{k+1}=1$.

Suppose that $k$ is a big descent of $\pi$. Then $\pi_{k}=\max\{\pi_{k},\pi_{k+1},\dots,\pi_{n}\}$, so $w_{k}=0$. Let $j\in[n-1]$ be the position of $\pi_{k}-1$ in $\pi$, i.e., $\pi_{j}=\pi_{k}-1$. Then we must have $k<j$, or else $\pi_{j}\pi_{k}\pi_{k+1}$ would be an occurrence of 231. Moreover, because $k$ is a big descent of $\pi$, we have $k+1<j$ and $\pi_{k+1}<\pi_{j}$. It follows that $\pi_{k+1}=\min\{\pi_{k+1},\pi_{k+2},\dots,\pi_{n}\}$, so $w_{k+1}=1$.

Conversely, suppose that $w_{k}=0$ and $w_{k+1}=1$. Since $\pi_{k}=\max\{\pi_{k},\pi_{k+1},\dots,\pi_{n}\}$, we know that $k$ is a descent of $\pi$. We have $k+1\leq\left|w\right|=n-1$ by assumption, so $n\geq k+2$\textemdash i.e., there is at least one letter in $\pi$ after $\pi_{k+1}$. Now, we see that $\pi_{k}=\max\{\pi_{k},\pi_{k+1},\dots,\pi_{n}\}$ and $\pi_{k+1}=\min\{\pi_{k+1},\pi_{k+2},\dots,\pi_{n}\}$ together imply $\pi_{k+1}\neq\pi_{k}-1$, as the letters appearing after $\pi_{k+1}$ are less than $\pi_{k}$ but greater than $\pi_{k+1}$. We conclude that $k$ is a big descent of $\pi$.
\end{proof}

It is well known that, for all $n\geq1$ and $k\geq0$, the number of binary words of length $n$ with $k$ occurrences of the factor $01$ is ${n+1 \choose 2k+1}$; see, e.g., \cite{Carlitz1977}. Theorem \ref{t-213-231} is now an immediate consequence of this formula along with Lemma~\ref{l-213-231}.

\subsection{The patterns 213 and 312}

Together with Theorem \ref{t-213-231}, the next result gives us our first non-trivial $\bdes$-Wilf equivalence.

\begin{thm} \label{t-213-312}
For all $n\geq1$ and $k\geq0$, we have 
\[
b_{n,k}(213,312)={n \choose 2k+1}.
\]
\end{thm}

Observe that permutations in $\mathfrak{S}_{n}(213,312)$ have the following structure: the letters appearing before $n$ must be increasing, and those appearing after $n$ must be decreasing. Thus, a permutation $\pi\in\mathfrak{S}_{n}(213,312)$ is completely determined by $n-1$ choices: for each letter $1,2,\dots,n-1$, we choose whether the letter is placed before $n$ or after $n$. This allows us to define the bijection $\phi_{213,312}\colon\mathfrak{S}_{n}(213,312)\rightarrow\mathcal{W}_{n-1}$ by 
\[
\phi_{213,231}(\pi)=w_{1}w_{2}\dots w_{n-1}\in\mathcal{W}_{n-1}
\]
where
\[
w_{k}=\begin{cases}
1, & \text{if the letter }k\text{ appears before }n\text{ in \ensuremath{\pi},}\\
0, & \text{otherwise,}
\end{cases}
\]
for all $k\in[n-1]$. Again, this description of $\phi_{213,312}$ was given by Bukata et al.\ \cite{Bukata2019}, but it is equivalent to an earlier bijection due to Simion and Schmidt \cite{Simion1985}.
\begin{lem} \label{l-213-312}
For all $n\geq1$ and $\pi\in\mathfrak{S}_{n}(213,312)$, we have $\bdes(\pi)=\occ_{01}(\phi_{213,312}(\pi))$.
\end{lem}

\begin{proof}
Let $\pi\in\mathfrak{S}_{n}(213,312)$ and $w=w_{1}w_{2}\dots w_{n-1}=\phi_{213,312}(\pi)$. We prove that $k\in[n-1]$ is a big descent of $\pi$ if and only if $w_{\pi_{k+1}}=0$ and $w_{\pi_{k+1}+1}=1$.

Suppose that $k\in[n-1]$ is a big descent of $\pi$, so that $\pi_{k}>\pi_{k+1}+1$. In particular, $\pi_{k+1}$ is part of the decreasing subsequence which follows $n$, so $w_{\pi_{k+1}}=0$. Furthermore, $\pi_{k+1}+1$ cannot be part of this decreasing subsequence, as that would imply that $\pi_{k+1}+1$ is positioned between $\pi_{k}$ and $\pi_{k+1}$. This means that $\pi_{k+1}+1$ is part of the increasing subsequence which precedes $n$, so $w_{\pi_{k+1}+1}=1$.

Conversely, suppose that $w_{\pi_{k+1}}=0$ and $w_{\pi_{k+1}+1}=1$. Then $\pi_{k+1}$ is part of the decreasing subsequence which follows $n$, so $\pi_{k}>\pi_{k+1}$. In addition, we cannot have $\pi_{k}=\pi_{k+1}+1$, since $\pi_{k+1}+1$ is part of the increasing subsequence which precedes $n$. Therefore, $k$ is a big descent of $\pi$.
\end{proof}

Lemma \ref{l-213-312} shows that the distribution of $\bdes$ over $\mathfrak{S}_{n}(213,312)$ is 
the same as that of $\occ_{01}$ over $\mathcal{W}_{n-1}$. As in Section \ref{ss-213-231}, this proves Theorem \ref{t-213-312}.
Combining Lemmas \ref{l-213-231} and \ref{l-213-312}, we see that the composition $\phi_{213,312}^{-1}\circ\phi_{213,231}$ provides a $\bdes$-preserving bijection between $\mathfrak{S}_{n}(213,231)$ and $\mathfrak{S}_{n}(213,312)$.

\subsection{\label{ss-123-132}The patterns 123 and 132}

Next we give a formula for the generating function $B(t,x;123,132)$.
\begin{thm} \label{t-123-132}
We have 
\[
B(t,x;123,132)=\frac{1-x+(1-t)x^{2}-(1-t)^{2}x^{3}}{1-2x+(1-t)x^{2}+t(1-t)x^{3}}.
\]
\end{thm}

Table \ref{tb-123-132} displays the first ten polynomials $B_{n}(t;123,132)$.

\begin{table}
\begin{centering}
\begin{tabular}{|c|c|c|c|c|}
\cline{1-2} \cline{2-2} \cline{4-5} \cline{5-5} 
$n$ & $B_{n}(t;123,132)$ & \quad{} & $n$ & $B_{n}(t;123,132)$\tabularnewline
\cline{1-2} \cline{2-2} \cline{4-5} \cline{5-5} 
$0$ & $1$ &  & $5$ & $2+8t+6t^{2}$\tabularnewline
\cline{1-2} \cline{2-2} \cline{4-5} \cline{5-5} 
$1$ & $1$ &  & $6$ & $2+11t+16t^{2}+3t^{3}$\tabularnewline
\cline{1-2} \cline{2-2} \cline{4-5} \cline{5-5} 
$2$ & $2$ &  & $7$ & $2+14t+31t^{2}+16t^{3}+t^{4}$\tabularnewline
\cline{1-2} \cline{2-2} \cline{4-5} \cline{5-5} 
$3$ & $2+2t$ &  & $8$ & $2+17t+51t^{2}+47t^{3}+11t^{4}$\tabularnewline
\cline{1-2} \cline{2-2} \cline{4-5} \cline{5-5} 
$4$ & $2+5t+t^{2}$ &  & $9$ & $2+20t+76t^{2}+104t^{3}+50t^{4}+4t^{5}$\tabularnewline
\cline{1-2} \cline{2-2} \cline{4-5} \cline{5-5} 
\end{tabular}
\par\end{centering}
\caption{\label{tb-123-132}Distribution of $\bdes$ over $\mathfrak{S}_{n}(123,132)$ for $0\leq n\leq9$.}
\end{table}

To prove this theorem, recall that, given $\pi\in\mathfrak{S}_{n}$ and $k\in[n]$, the letter $\pi_{k}$ is a \textit{right-to-left maximum} of $\pi$ if $\pi_{k}>\pi_{j}$ for all $j>k$. Let $\RLmax(\pi)$ denote the set of right-to-left maxima of $\pi$, and note that $n\in\RLmax(\pi)$. For example, we have $\RLmax(2745361)=\{1,6,7\}$.

Simion and Schmidt showed that the map $\pi\mapsto\RLmax(\pi)\backslash\{n\}$ is a bijection from $\mathfrak{S}_{n}(123,132)$ to $2^{[n-1]}$; see the proof of \cite[Proposition 12]{Simion1985}. In other words, given any set $S\subseteq[n-1]$, there is a unique permutation $\pi\in\mathfrak{S}_{n}(123,132)$ such that $\RLmax(\pi)=S\cup\{n\}$. Thus, $\pi\mapsto\RLmax(\pi)$ is a bijection from $\mathfrak{S}_{n}(123,132)$ to the set of subsets of $[n]$ that contain~$n$. Next we describe the inverse of this map.

Given $a,b\in\mathbb{Z}$ satisfying $a<b$, let $\overleftarrow{(a,b)}$ denote the decreasing word
\[
\overleftarrow{(a,b)}\coloneqq(b-1)\:(b-2)\:\dots\:(a+1).
\]
Then the unique permutation $\pi\in\mathfrak{S}_{n}(123,132)$ with $\RLmax(\pi)=\{s_1,s_2,\dots,s_m\}$, where $s_{1}<s_{2}<\cdots<s_{m}=n$, is the concatenation
\begin{equation}
\pi=\overleftarrow{(s_{m-1},n)}\:n\:\overleftarrow{(s_{m-2},s_{m-1})}\:s_{m-1}\:\dots\:\overleftarrow{(s_{1},s_{2})\:}s_{2}\:\overleftarrow{(0,s_{1})}\:s_{1}.\label{e-123-132-rlmax}
\end{equation}
For example, if $\{2,5,9\}$ is the set of right-to-left maxima of a permutation $\pi$ avoiding $123$ and $132$, then $\pi=876943512$. Equation \eqref{e-123-132-rlmax} demonstrates that big descents of permutations in $\mathfrak{S}_{n}(123,132)$ can only occur at positions of right-to-left maxima.

Recall that $\mathcal{W}_{n}^{(1)}$ is the set of words in $\mathcal{W}_{n}$ that end with the letter 1. We turn the bijection $\pi\mapsto\RLmax(\pi)$ into a bijection from $\mathfrak{S}_{n}(123,132)$ to $\mathcal{W}_{n}^{(1)}$ by defining 
\[
\phi_{123,132}(\pi)=w_{1}w_{2}\dots w_{n}
\]
where
\[
w_{k}=\begin{cases}
1, & \text{if }k\in\RLmax(\pi),\\
0, & \text{otherwise,}
\end{cases}
\]
for each $k\in[n]$. Let us show that this bijection turns big descents into 10- and 011-factors. In the proof of the next lemma, we use the notation $s_i^\uparrow\coloneqq s_{i+1}$ and $s_{i+1}^\downarrow\coloneqq s_i$ to indicate that, when listing the right-to-left maxima of $\pi$ in increasing order, $s_{i+1}$ comes immediately after~$s_i$.

\begin{lem} \label{l-123-132}
For all $n\geq1$ and $\pi\in\mathfrak{S}_{n}(123,132)$, we have 
\[
\bdes(\pi)=(\occ_{10}+\occ_{011})(\phi_{123,132}(\pi)).
\]
\end{lem}

\begin{proof}
Let $\pi\in\mathfrak{S}_{n}(123,132)$ and $w=w_{1}w_{2}\dots w_{n}=\phi_{123,132}(\pi)$. We prove two statements, from which the desired conclusion follows:
\begin{enumerate}
\item [\normalfont{(I)}]If $k\in[n-1]$ is a big descent of $\pi$ and $\pi_{k}-1$ appears before $\pi_{k}$ in $\pi$, then $w$ has a 10-factor starting at position $y=\pi_k^\downarrow$.
Conversely, if $w$ has a 10-factor starting at position $y\in[n-2]$ and $\pi_{k}=y^\uparrow$, 
then $k$ is a big descent of $\pi$ and $\pi_{k}-1$ appears before $\pi_{k}$ in $\pi$.
\item [\normalfont{(II)}]We have that $k\in[n-1]$ is a big descent of $\pi$ and $\pi_{k}-1$ appears after $\pi_{k}$ in $\pi$ if and only if $w$ has a 011-factor starting at position $\pi_{k}-2$.
\end{enumerate}
First, suppose that $k\in[n-1]$ is a big descent of $\pi$ and $\pi_{k}-1$ appears before $\pi_{k}$ in $\pi$. From \eqref{e-123-132-rlmax}, we see that $\pi_{k}$ is a right-to-left maximum and that the decreasing subsequence of non-right-to-left maxima preceding $\pi_{k}$ has length at least one (since it contains $\pi_{k}-1$). Therefore, if $y=\pi_{k}^\downarrow$, then $w_{y}=1$ and $w_{y+1}=0$.

Conversely, suppose that $w_{y}=1$ and $w_{y+1}=0$ for some $y\in[n-2]$, and let $\pi_{k}=y^\uparrow$.
By \eqref{e-123-132-rlmax}, $\pi_{k}-1$ appearing after $\pi_{k}$ in $\pi$ would imply $\pi_{k}-1=y$, which contradicts the fact that $y+1$ is not a right-to-left maximum. We deduce that $\pi_{k}-1$ appears before $\pi_{k}$. Furthermore, by \eqref{e-123-132-rlmax}, we see that $k<n$ (as $y$ appears after $\pi_{k}$), and since $\pi_{k}$ is a right-to-left maximum and $\pi_{k+1}\neq\pi_{k}-1$ (as $\pi_{k}-1$ appears before $\pi_{k}$), we deduce that $\pi_{k+1}<\pi_{k}-1$, thus completing the proof of (I).

Next, suppose that $k$ is a big descent of $\pi$ and $\pi_{k}-1$ appears after $\pi_{k}$ in $\pi$. From \eqref{e-123-132-rlmax}, we see that both $\pi_{k}$ and $\pi_{k}-1$ are right-to-left maxima of $\pi$. Because $k$ is a big descent, we also know that the decreasing subsequence of non-right-to-left maxima between $\pi_{k}$ and $\pi_{k}-1$ is nonempty, and in particular must contain $\pi_{k}-2$. Hence, we have $w_{\pi_{k}-2}=0$, $w_{\pi_{k}-1}=1$, and $w_{\pi_{k}}=1$. The converse follows from similar reasoning, proving~(II).
\end{proof}

The last step in the proof of Theorem \ref{t-123-132} is the enumeration of words in $\mathcal{W}_{n}^{(1)}$ with respect to the number of $10$- and $011$-factors. For any fixed binary word $u$, denote by $\mathcal{W}_n^{(u)}$ the set of binary words of length $n$ that end with $u$, and let 
$$W^{(u)}\coloneqq W^{(u)}(t,x)=\sum_{n=0}^{\infty}\sum_{w\in\mathcal{W}_{n}^{(u)}}t^{(\occ_{10}+\occ_{011})(w)}x^{n}.$$

\begin{prop} \label{p-10-011-gf}
We have 
\[
W^{(1)}=\frac{x(1-(1-t)x^{2})}{1-2x+(1-t)x^{2}+t(1-t)x^{3}}.
\]
\end{prop}

\begin{proof}
We claim that $W^{(1)}$ , $W^{(0)}$, $W^{(01)}$, and $W^{(11)}$ satisfy the system of equations
\begin{align*}
    W^{(1)}&=x+W^{(01)}+W^{(11)}, &W^{(0)}&=(1+W^{(0)}+t W^{(1)})x, \\
    W^{(01)}&=W^{(0)}x, \quad\text{and} &W^{(11)}&=(x+tW^{(01)}+W^{(11)})x.
\end{align*}

The equation for $W^{(1)}$ is clear since any binary word ending with $1$ is either the word $1$ itself (contributing $x$ to the generating function), or otherwise it ends with $01$ or with $11$.
For $W^{(0)}$, note that a binary word ending with $0$ is of the form $w0$, where $w$ is an arbitrary binary word; if $w$ ends with $1$, then $w0$ has an additional $10$-factor.
The equation for $W^{(01)}$ is straightforward, since appending a $1$ to a word ending with $0$ does not create any $10$- or $011$-factors.
For $W^{(11)}$, note that a binary word ending with $11$ is of the form $w1$, where $w$ is word ending with $1$; if $w$ ends with $01$, then $w1$ has an additional $011$-factor.

Solving the system, we obtain the stated expression for $W^{(1)}$.
\end{proof}

Lemma \ref{l-123-132} implies that $B(t,x;123,132)=1+W^{(1)}$, where the term $1$ accounts for the empty permutation. This proves Theorem~\ref{t-123-132}.

\subsection{\label{ss-132-213}The patterns 132 and 213}

In this section we will show that $B(t,x;132,213)=B(t,x;123,132)$, 
giving our second non-trivial $\bdes$-Wilf equivalence.

\begin{thm} \label{t-132-213}
We have 
\[
B(t,x;132,213)=\frac{1-x+(1-t)x^{2}-(1-t)^{2}x^{3}}{1-2x+(1-t)x^{2}+t(1-t)x^{3}}.
\]
\end{thm}

The proof of this theorem is similar to that of Theorem \ref{t-123-132}.
As shown by Simion and Schmidt \cite[proof of Proposition 12]{Simion1985}, the map $\pi\mapsto\RLmax(\pi)\backslash\{n\}$ is also a bijection from $\mathfrak{S}_{n}(132,213)$ to $2^{[n-1]}$, hence $\pi\mapsto\RLmax(\pi)$ is a bijection from $\mathfrak{S}_{n}(132,213)$ to the set of subsets of $[n]$ that contain $n$. Given $a,b\in\mathbb{Z}$ satisfying $a<b$, let $\overrightarrow{(a,b)}$ denote the increasing word
\[
\overrightarrow{(a,b)}\coloneqq(a+1)\:(a+2)\:\dots\:(b-1).
\]
Then the unique permutation $\pi\in\mathfrak{S}_{n}(132,213)$ satisfying 
$\RLmax(\pi)=\{s_1,s_2,\dots,s_m\}$, where $s_{1}<s_{2}<\cdots<s_{m}=n$, is the concatenation
\begin{equation}
\pi=\overrightarrow{(s_{m-1},n)}\:n\:\overrightarrow{(s_{m-2},s_{m-1})}\:s_{m-1}\:\dots\:\overrightarrow{(s_{1},s_{2})}\:s_{2}\:\overrightarrow{(0,s_{1})}\:s_{1}.\label{e-132-213-rlmax}
\end{equation}
Comparing this with \eqref{e-123-132-rlmax}, the only difference is that each sequence of non-right-to-left maxima is increasing rather than decreasing.

We define $\phi_{132,213}\colon\mathfrak{S}_{n}(132,213)\rightarrow\mathcal{W}_{n}^{(1)}$ in the same way as $\phi_{123,132}$, by having the positions of the ones in the word record the right-to-left maxima of the permutation. The proof of the next lemma is very similar to that of Lemma \ref{l-123-132}, so we omit it.

\begin{lem} \label{l-132-213}
For all $n\geq1$ and $\pi\in\mathfrak{S}_{n}(132,213)$, we have 
\[
\bdes(\pi)=(\occ_{10}+\occ_{011})(\phi_{132,213}(\pi)).
\]
\end{lem}

Lemmas \ref{l-123-132} and \ref{l-132-213} imply that $\phi_{132,213}^{-1}\circ\phi_{123,132}$ is a $\bdes$-preserving bijection between
$\mathfrak{S}_{n}(123,132)$ and $\mathfrak{S}_{n}(132,213)$, so Theorem \ref{t-132-213}  follows from Theorem \ref{t-123-132}.

\subsection{The patterns 231 and 321}

Our next result gives a formula for the generating function $B(t,x;231,321)$.

\begin{thm} \label{t-231-321}
We have 
\[
B(t,x;231,321)=\frac{1-x}{1-2x+(1-t)x^{3}}.
\]
\end{thm}

See Table \ref{tb-231-321} for the first ten polynomials $B_{n}(t;231,321)$.

\begin{table}
\begin{centering}
\begin{tabular}{|c|c|c|c|c|}
\cline{1-2} \cline{2-2} \cline{4-5} \cline{5-5} 
$n$ & $B_{n}(t;231,321)$ & \quad{} & $n$ & $B_{n}(t;231,321)$\tabularnewline
\cline{1-2} \cline{2-2} \cline{4-5} \cline{5-5} 
$0$ & $1$ &  & $5$ & $8+8t$\tabularnewline
\cline{1-2} \cline{2-2} \cline{4-5} \cline{5-5} 
$1$ & $1$ &  & $6$ & $13+18t+t^{2}$\tabularnewline
\cline{1-2} \cline{2-2} \cline{4-5} \cline{5-5} 
$2$ & $2$ &  & $7$ & $21+38t+5t^{2}$\tabularnewline
\cline{1-2} \cline{2-2} \cline{4-5} \cline{5-5} 
$3$ & $3+t$ &  & $8$ & $34+76t+18t^{2}$\tabularnewline
\cline{1-2} \cline{2-2} \cline{4-5} \cline{5-5} 
$4$ & $5+3t$ &  & $9$ & $55+147t+53t^{2}+t^{3}$\tabularnewline
\cline{1-2} \cline{2-2} \cline{4-5} \cline{5-5} 
\end{tabular}
\par\end{centering}
\caption{\label{tb-231-321}Distribution of $\bdes$ over $\mathfrak{S}_{n}(231,321)$ for $0\leq n\leq9$}
\end{table}

The permutations in $\mathfrak{S}_{n}(231,321)$ are precisely the reverses of the permutations in $\mathfrak{S}_{n}(123,132)$, studied in Section \ref{ss-123-132}. As such, we get a bijection $\pi\mapsto\LRmax(\pi)$ from $\mathfrak{S}_{n}(231,321)$ to the set of subsets of $[n]$ containing $\{n\}$, where $\LRmax(\pi)$ denotes the set of \textit{left-to-right maxima} of $\pi$: letters $\pi_{k}$ satisfying $\pi_{j}<\pi_{k}$ for all $j<k$. As in \eqref{e-123-132-rlmax}, the unique permutation $\pi\in\mathfrak{S}_{n}(123,132)$ satisfying $\LRmax(\pi)=\{s_1,s_2,\dots,s_m\}$, where $s_{1}<s_{2}<\cdots<s_{m}=n$, is the concatenation
\begin{equation}
\pi=s_{1}\:\overrightarrow{(0,s_{1})}\:s_{2}\:\overrightarrow{(s_{1},s_{2})}\:\dots\:s_{m-1}\:\overrightarrow{(s_{m-2},s_{m-1})}\:n\:\overrightarrow{(s_{m-1},n)}.\label{e-231-321-lrmax}
\end{equation}
Moreover, we obtain a bijection $\phi_{231,321}\colon\mathfrak{S}_{n}(231,321)\rightarrow\mathcal{W}_{n}^{(1)}$ by defining
\[
\phi_{231,321}(\pi)=w_{1}w_{2}\dots w_{n}\in\mathcal{W}_{n}^{(1)}
\]
where
\[
w_{k}=\begin{cases}
1, & \text{if }k\in\LRmax(\pi),\\
0, & \text{otherwise,}
\end{cases}
\]
for each $k\in[n]$.

\begin{lem} \label{l-231-321}
For all $n\geq1$ and $\pi\in\mathfrak{S}_{n}(231,321)$, we have $\bdes(\pi)=\occ_{001}(\phi_{231,321}(\pi))$.
\end{lem}

\begin{proof}
Let $\pi\in\mathfrak{S}_{n}(231,321)$ and $w=w_{1}w_{2}\dots w_{n}=\phi_{231,321}(\pi)$. We claim that $k\in[n-1]$ is a big descent of $\pi$ if and only if $w_{\pi_{k}-2}=0$, $w_{\pi_{k}-1}=0$, and $w_{\pi_{k}}=1$.

Suppose that $k$ is a big descent of $\pi$. From \eqref{e-231-321-lrmax}, we see that $\pi_{k}$ is a left-to-right maximum of $\pi$, and that $\pi_{k+1}$ is part of the sequence of non-left-to-right maxima following $\pi_{k}$. Since $\pi_{k+1}<\pi_{k}-1$, both $\pi_{k}-2$ and $\pi_{k}-1$ are part of this sequence. Therefore, we have $w_{\pi_{k}-2}=0$, $w_{\pi_{k}-1}=0$, and $w_{\pi_{k}}=1$. The proof of the converse is similar.
\end{proof}

Given a binary word $w\in\mathcal{W}_{n}$, let $\run_{r}(w)$ denote the number of its maximal runs of 0's having length at least $r$. For example, we have $\run_{2}(0000110111001)=2$. Let 
\[
R^{(r)}(t,x)\coloneqq\sum_{n=0}^{\infty}\sum_{w\in\mathcal{W}_{n}}t^{\run_{r}(w)}x^{n}.
\]
It is known \cite[Chapter V, Equation (18)]{Flajolet2009} that 
\begin{equation}
R^{(r)}(t,x)=\frac{1-(1-t)x^{r}}{1-x}\cdot\frac{1}{1-x\frac{1-(1-t)x^{r}}{1-x}}\label{e-rungf}
\end{equation}
for all $r\geq2$. Given any $w\in\mathcal{W}_{n}^{(1)}$, if we let $w^{\prime}\in\mathcal{W}_{n-1}$ be the word obtained by removing the last letter, we have $\occ_{001}(w)=\run_{2}(w^{\prime})$. Thus, by Lemma \ref{l-231-321}, we obtain $B(t,x;231,321)$ by setting $r=2$ in \eqref{e-rungf}, multiplying by $x$, and adding 1; doing so proves Theorem \ref{t-231-321}.

\subsection{\label{ss-others}Other pairs of patterns}

Finally, we prove the formulas in the last four rows of Table \ref{tb-enum}. We begin with $\Pi=\{123,231\}$ and $\Pi=\{132,321\}$; permutations in $\mathfrak{S}_{n}(\Pi)$ for either of these $\Pi$ have at most one big descent.

\begin{prop} \label{p-123-231}
For all $n\geq1$, we have
\[
b_{n,k}(123,231)=\begin{cases}
n, & \text{if }k=0,\\
{n-1 \choose 2}, & \text{if }k=1,\\
0, & \text{otherwise.}
\end{cases}
\]
\end{prop}

\begin{proof}
First, we argue that no permutation in $\mathfrak{S}_{n}(123,231)$ has more than one big descent. Suppose otherwise; let $\pi\in\mathfrak{S}_{n}(123,231)$ have two big descents $i$ and $j$ with $i<j$. Since $\pi\in\mathfrak{S}_{n}(123,231)$, we may write $\pi=\sigma1\tau$ where $\sigma$ and $\tau$ are both decreasing. Note that neither $\pi_{i}$ nor $\pi_{j}$ can be in $\tau$. If one of them\textemdash say $\pi_{i}$\textemdash were in $\tau$, then $\pi_{i}-1$ must appear before $\pi_{i}$ in $\pi$ because $\tau$ is decreasing and $\pi_{i+1}<\pi_{i}-1$; then $\pi_{i}-1$, $\pi_{i}$, and $\pi_{i+1}$ would be an occurrence of 231. Thus, both $\pi_{i}$ and $\pi_{j}$ are in $\sigma$. Next, since $\sigma$ is decreasing and $i$ and $j$ are big descents, both $\pi_{i}-1$ and $\pi_{j}-1$ are in $\tau$. And since $i<j$, we have $\pi_{i}>\pi_{j}$, which implies $\pi_{i}-1>\pi_{j}-1$. This means that $\pi_{i}-1$ appears before $\pi_{j}-1$ in $\tau$, since $\tau$ is decreasing. It follows that $\pi_{j}$, $\pi_{i}-1$, and $\pi_{j}-1$ form an occurrence of 231 in $\pi$, a contradiction.

Now, we claim that $b_{n,0}(123,231)=n$. Suppose that $\pi\in\mathfrak{S}_{n}(123,231)$ has no big descents. Let us write $\pi=\sigma1\tau$ where $\sigma$ and $\tau$ are both decreasing, as above. Note that if $\pi_{1}=j$, then we must have 
\[
\sigma=j\:(j-1)\:\dots\:2\qquad\text{and}\qquad\tau=n\:(n-1)\:\dots\:(j+1)
\]
in order for $\pi$ to have no big descents. Therefore, the choice of $\pi_{1}$ completely determines $\pi$, and since there are $n$ choices for $\pi_{1}$, we have $b_{n,0}(123,231)=n$. Hence, the desired result follows from the formula 
\begin{equation}
\left|\mathfrak{S}_{n}(123,231)\right|={n \choose 2}+1\label{e-123-231ss},
\end{equation}
due to Simion and Schmidt 
\cite[Lemma 5(e) and Proposition 11]{Simion1985}, and the calculation
${n \choose 2}+1-n={n-1 \choose 2}$.
\end{proof}

\begin{prop} \label{p-132-321}
For all $n\geq2$, we have
\[
b_{n,k}(132,321)=\begin{cases}
2, & \text{if }k=0,\\
{n \choose 2}-1, & \text{if }k=1,\\
0, & \text{otherwise.}
\end{cases}
\]
\end{prop}

\begin{proof}
Let $\pi\in\mathfrak{S}_{n}(132,321)$ where $n\geq2$. We can write $\pi=\sigma1\tau$ where $\sigma$ and $\tau$ are both increasing. Thus $\pi$ has no descents if $\sigma$ is empty, and has exactly one descent otherwise. In particular, $\pi$ has no big descents if and only if $\sigma$ is empty or $\sigma=2$. Hence, we have $b_{n,0}=2$, and the result follows from \eqref{e-123-231ss} and reversal symmetry.
\end{proof}

The cases $\Pi=\{231,312\}$ and $\Pi=\{123,321\}$ are trivial, but we state these results for the sake of completeness.

\begin{prop} \label{p-231-312}
For all $n\geq1$, we have
\[
b_{n,k}(231,312)=\begin{cases}
2^{n-1}, & \text{if }k=0,\\
0, & \text{otherwise.}
\end{cases}
\]
\end{prop}

\begin{proof}
Since permutations simultaneously avoiding 231 and 312 cannot have big descents, the result follows from the formula
\[
\left|\mathfrak{S}_{n}(231,312)\right|=2^{n-1}
\]
due to Simion and Schmidt
\cite[Lemma 5(b) and Proposition 8]{Simion1985}.
\end{proof}

\begin{prop} \label{p-123-321}
We have
\begin{alignat*}{1}
B(t,x;123,321) & =1+x+2x^{2}+(2+2t)x^{3}+(1+2t+t^{2})x^{4}.
\end{alignat*}
\end{prop}

\begin{proof}
We know from the celebrated Erd\H{o}s\textendash Szekeres theorem \cite{Erdoes1935} that $\mathfrak{S}_{n}(123,321)=\emptyset$ for all $n\geq5$, and the distribution of $\bdes$ over $\mathfrak{S}_{n}(123,321)$ for $n\le4$ immediately gives the above equation.
\end{proof}

\section{\label{s-future}Future directions of research}

We conclude this paper by proposing some directions for further study.

\subsection{The pattern 123, other \texorpdfstring{$\Pi\subseteq\mathfrak{S}_{3}$}{length 3 pattern sets}, and \texorpdfstring{$r$}{r}-descents}

As part of our investigation on 123-avoiding permutations, we found that the number of permutations in $\mathfrak{S}_{n}(123)$ with $k$ big right descents is the Narayana number $N_{n,k}$ (Theorem~\ref{t-123-nar}), but our proof of this fact is complicated and sheds little combinatorial insight into why the Narayana numbers appear. To that end, we ask for direct bijective proofs for this result and the analogous result for (ordinary) big descents (Theorem~\ref{t-123}). 

\begin{problem} \label{p-bij}
Find bijective proofs for Theorems \ref{t-123}--\ref{t-123-nar}.
\end{problem}

Among other things, the Narayana numbers are known to count Dyck paths by peaks and 123-avoiding permutations by right-to-left maxima; these may provide potential avenues for Problem~\ref{p-bij}.

In this paper we have only considered the case where the pattern set $\Pi\subseteq\mathfrak{S}_{3}$ is of size 1 or 2, but one also could ask the following.

\begin{problem}
Describe the distribution of $\bdes$ over $\mathfrak{S}_{n}(\Pi)$ for pattern sets $\Pi\subseteq\mathfrak{S}_{3}$ with $\left|\Pi\right|\geq3$.
\end{problem}

Finally, recall that big descents are only a special case ($r=1$) of the more general notion of $r$-descents.

\begin{problem}
Describe the distribution of $\des_{r}$ over families of pattern-avoiding permutations.
\end{problem}

\subsection{Real-rootedness and log-concavity}

Given a set $\Pi$ of patterns, it is natural to ask whether the sequence $\{b_{n,k}(\Pi)\}_{0\leq k\leq n}$ is unimodal for all $n\geq0$, meaning that there exists $j\in[n]$ for which
\[
b_{n,0}(\Pi)\leq b_{n,1}(\Pi)\leq\cdots\leq b_{n,j}(\Pi)\geq\cdots\geq b_{n,n-1}(\Pi)\geq b_{n,n}(\Pi).
\]
Unimodality is closely related to the properties of real-rootedness and log-concavity. A polynomial with real coefficients is called \textit{real-rooted} if all of its roots are real, and a sequence $\{a_{k}\}_{0\leq k\leq n}$ is said to be \textit{log-concave} if $a_{k}^{2}\geq a_{k-1}a_{k+1}$ for all $1\leq k\leq n-1$. It is well known that if a polynomial with non-negative coefficients is real-rooted, then the sequence of its coefficients is log-concave, which in turn implies unimodality of the coefficient sequence \cite[Lemma 1.1]{Braenden2015}. 

Empirical evidence suggests that the polynomials $B_{n}(t;\Pi)$ are real-rooted for almost every $\Pi$ considered in this paper.

\begin{conjecture} \label{cj-realroots}
Let $\Pi\subseteq\mathfrak{S}_{3}$ with $\left|\Pi\right|=1$ or $\left|\Pi\right|=2$, and suppose that $\Pi\notin\left[\{123,132\}\right]_{\bdes}$. Then $B_{n}(t;\Pi)$ is real-rooted for all $n\geq0$.
\end{conjecture}

The polynomials $B_{n}(t;\Pi)$ for the pattern sets $\Pi$ considered in Section \ref{ss-others} are real-rooted for trivial reasons. Additionally, the real-rootednesss of $B_{n}(t;123)$ appears to be known, according to Peter Bala's comment in \cite[A108838]{oeis}. Below we will prove the real-rootedness of $B_{n}(t;231)$.

\begin{thm}
The polynomial $B_{n}(t;231)$ is real-rooted for all $n\geq0$.
\end{thm}

\begin{proof}
The result is trivial for $n=0$, so assume $n\geq1$. Let 
\[
A_{n}(t;231)\coloneqq\sum_{\pi\in\mathfrak{S}_{n}(231)}t^{\des(\pi)}\quad\text{and}\quad P_{n}(t;231)\coloneqq\sum_{\pi\in\mathfrak{S}_{n}(231)}t^{\pk(\pi)}
\]
where $\pk$ is the peak number statistic as defined in Section \ref{ss-231}. The polynomials $A_{n}(t;231)$ are the Narayana polynomials, which are known to be real-rooted; see, e.g., \cite{Liu2007}. It follows from the work of Br\"and\'en \cite[Corollaries 3.2 and 4.2]{Braenden2008} that 
\begin{equation}
A_{n}(t;231)=\left(\frac{1+t}{2}\right)^{n-1}P_{n}\left(\frac{4t}{(1+t)^{2}};231\right)\label{e-branden}
\end{equation}
for all $n\geq1$. Furthermore, Stembridge \cite[Proposition 4.4]{Stembridge1997} proved that if $f(t)$ is a polynomial with non-negative real coefficients, then $f$ is real-rooted if and only if $f(4t/(1+t)^{2})$ is real-rooted; applying this result to \eqref{e-branden} proves that $P_{n}(t;231)$ is real-rooted for every $n\geq1$. But this is precisely the result we desire, by the equidistribution of $\pk$ and $\bdes$ over $\mathfrak{S}_{n}(231)$ (Theorem \ref{t-231joint}).
\end{proof}

Even though the polynomials $B_{n}(t;123,132)$ are not real-rooted, their coefficients seem to be log-concave.

\begin{conjecture}
Let $\Pi\subseteq\mathfrak{S}_{3}$ with $\left|\Pi\right|=1$ or $\left|\Pi\right|=2$. The sequence $\{b_{n,k}(\Pi)\}_{0\leq k\leq n}$ is log-concave for all $n\geq0$.
\end{conjecture}

\subsection{Big descent sets and quasisymmetric functions}

This subsection assumes basic familiarity with symmetric and quasisymmetric functions. We refer the reader to \cite[Chapters 7--8]{Sagan2020b} and \cite[Chapter 7]{Stanley2024} for definitions.

In addition to counting pattern-avoiding permutations by the \textit{number} of big descents, one can more generally consider the distribution of the \textit{set} of big descents. Given $\pi\in\mathfrak{S}_{n}$, let $\Bdes(\pi)$ denote the set of big descents of $\pi$. Since $\Bdes(\pi)\subseteq[n-1]$ for each $\pi\in\mathfrak{S}_{n}$, we can encode the distribution of $\Bdes$ over a set of permutations with a quasisymmetric function. To that end, define 
\[
Q_{n}^{(1)}(\Pi)\coloneqq\sum_{\pi\in\mathfrak{S}_{n}(\Pi)}F_{n,\Bdes(\pi)}
\]
where $F_{n,S}$ is the \textit{fundamental quasisymmetric function} associated to $S\subseteq[n-1]$. Then $\Pi$ and $\Pi^{\prime}$ are $\Bdes$-Wilf equivalent if and only if $Q_{n}^{(1)}(\Pi)=Q_{n}^{(1)}(\Pi^{\prime})$ for all $n\geq0$. 

It is natural to ask whether such quasisymmetric functions are in fact symmetric, and if so, whether they are \textit{Schur positive}, meaning that their expansion in terms of Schur functions has non-negative coefficients. Hamaker, Pawlowski, and Sagan \cite{Hamaker2020} investigated these questions for the quasisymmetric functions 
\[
Q_{n}(\Pi)\coloneqq\sum_{\pi\in\mathfrak{S}_{n}(\Pi)}F_{n,\Des(\pi)}
\]
where $\Des(\pi)$ is the descent set of $\pi$, and they proved a collection of results giving cases when $Q_{n}(\Pi)$ is symmetric and Schur positive. We pose the analogous questions for the quasisymmetric functions $Q_{n}^{(1)}(\Pi)$.

\begin{question} \label{q-quasisym}
For which pattern sets $\Pi$ is $Q_{n}^{(1)}(\Pi)$ symmetric for all $n\geq0$? For which $\Pi$ is $Q_{n}^{(1)}(\Pi)$ Schur positive for all $n\geq0$?
\end{question}

Taking $\Pi = \emptyset$, we have $Q_{n}^{(1)}(\emptyset) = \sum_{\pi\in\mathfrak{S}_{n}} F_{n,\Bdes(\pi)}$, for which symmetry and Schur positivity are already known.

\begin{thm} \label{t-bdessp}
For all $n\geq0$, the quasisymmetric function $Q_{n}^{(1)}(\emptyset)$ is symmetric and Schur positive.
\end{thm}

As pointed out to us by Hsin-Chieh Liao, this result is implicit in the work of Shareshian and Wachs \cite{Shareshian2010, Shareshian2016} on Eulerian and chromatic quasisymmetric functions. In short, it follows from \cite[Theorem 1.2]{Shareshian2010} and \cite[Theorem 3.1 and Equation (9.7)]{Shareshian2016} that $Q_{n}^{(1)}(\emptyset) = \sum_{j\geq 0} Q_{n,j}$, where the $Q_{n,j}$ are \textit{Eulerian quasisymmetric functions}, which are known to be Schur positive symmetric functions \cite[Theorem 5.1]{Shareshian2010}, so the sums $\sum_{j\geq 0} Q_{n,j}$ are as well.

We also pose the following conjecture.

\begin{conjecture}
For all $n\geq0$ and $m\geq1$, the quasisymmetric function $Q_{n}^{(1)}(12\dots m)$ is symmetric and Schur positive.
\end{conjecture}

See Tables \ref{tb-schurempty}\textendash \ref{tb-schur4}  for the Schur expansions of $Q_{n}^{(1)}(\emptyset)$, $Q_{n}^{(1)}(123)$, and $Q_{n}^{(1)}(1234)$ for $n\le7$.

More generally, for any $r\ge0$, one can define
\[
Q_{n}^{(r)}(\Pi)\coloneqq\sum_{\pi\in\mathfrak{S}_{n}(\Pi)}F_{n,\Des_{r}(\pi)}
\]
where $\Des_{r}(\pi)$ is the set of $r$-descents of $\pi$. Then the quasisymmetric functions $Q_{n}(\Pi)$ studied in \cite{Hamaker2020} are $Q_{n}^{(0)}(\Pi)$, and we can ask Question \ref{q-quasisym} more generally for $Q_{n}^{(r)}(\Pi)$. Shareshian and Wachs's work on chromatic quasisymmetric functions \cite{Shareshian2016} may provide a fruitful avenue for studying $Q_{n}^{(r)}(\emptyset)$.

\vspace{20bp}

\noindent \textbf{Acknowledgements.} Most of the work in this paper was done as part of the second author's honors thesis project at Davidson College (under the supervision of the third author); we thank Heather Smith Blake, Hammurabi Mendes, Donna Molinek, and Ana Wright for serving on the second author's thesis committee and offering helpful feedback on his work. We also thank Kyle Petersen for discussions which led to the idea for this project, as well as Hsin-Chieh Liao, Lara Pudwell, and Bruce Sagan for helpful e-mail correspondence. The third author was partially supported by NSF grant DMS-2316181.

\vspace{20bp}

\begin{table}[H]
\begin{centering}
\begin{tabular}{|>{\centering}m{8bp}|>{\raggedright}m{4.5in}|}
\hline 
$n$ & $Q_{n}^{(1)}(\emptyset)$\tabularnewline
\hline 
$0$ & $s_{()}$\tabularnewline
\hline 
$1$ & $s_{(1)}$\tabularnewline
\hline 
$2$ & $2s_{(2)}$\tabularnewline
\hline 
$3$ & $s_{(2,1)}+4s_{(3)}$\tabularnewline
\hline 
$4$ & $2s_{(2,2)}+4s_{(3,1)}+8s_{(4)}$\tabularnewline
\hline 
$5$ & $s_{(2,2,1)}+s_{(3,1,1)}+9s_{(3,2)}+12s_{(4,1)}+16s_{(5)}$\tabularnewline
\hline 
$6$ & $2s_{(2,2,2)}+8s_{(3,2,1)}+12s_{(3,3)}+6s_{(4,1,1)}+30s_{(4,2)}+32s_{(5,1)}+32s_{(6)}$\tabularnewline
\hline 
$7$ & $s_{(2,2,2,1)}+2s_{(3,2,1,1)}+14s_{(3,2,2)}+18s_{(3,3,1)}+s_{(4,1,1,1)}$\\
$\vphantom{a^{a^{a}}}\qquad+38s_{(4,2,1)}+57s_{(4,3)}+24s_{(5,1,1)}+88s_{(5,2)}+80s_{(6,1)}+64s_{(7)}$\tabularnewline
\hline 
\end{tabular}
\par\end{centering}
\caption{\label{tb-schurempty}Schur expansions of $Q_{n}^{(1)}(\emptyset)$
for $0\leq n\leq7$.}
\end{table}

\begin{table}[H]
\begin{centering}
\begin{tabular}{|>{\centering}p{8bp}|>{\raggedright}p{4.5in}|}
\hline 
$n$ & $Q_{n}^{(1)}(123)$\tabularnewline
\hline 
$0$ & $s_{()}$\tabularnewline
\hline 
$1$ & $s_{(1)}$\tabularnewline
\hline 
$2$ & $2s_{(2)}$\tabularnewline
\hline 
$3$ & $s_{(2,1)}+3s_{(3)}$\tabularnewline
\hline 
$4$ & $2s_{(2,2)}+2s_{(3,1)}+4s_{(4)}$\tabularnewline
\hline 
$5$ & $s_{(2,2,1)}+4s_{(3,2)}+3s_{(4,1)}+5s_{(5)}$\tabularnewline
\hline 
$6$ & $2s_{(2,2,2)}+2s_{(3,2,1)}+2s_{(3,3)}+6s_{(4,2)}+4s_{(5,1)}+6s_{(6)}$\tabularnewline
\hline 
$7$ & $s_{(2,2,2,1)}+4s_{(3,2,2)}+s_{(3,3,1)}+3s_{(4,2,1)}+4s_{(4,3)}+8s_{(5,2)}+5s_{(6,1)}+7s_{(7)}$\tabularnewline
\hline 
\end{tabular}
\par\end{centering}
\caption{\label{tb-schur3}Schur expansions of $Q_{n}^{(1)}(123)$ for $0\leq n\leq7$.}
\end{table}

\begin{table}[H]
\begin{centering}
\begin{tabular}{|>{\centering}m{8bp}|>{\raggedright}m{4.5in}|}
\hline 
$n$ & $Q_{n}^{(1)}(1234)$\tabularnewline
\hline 
$0$ & $s_{()}$\tabularnewline
\hline 
$1$ & $s_{(1)}$\tabularnewline
\hline 
$2$ & $2s_{(2)}$\tabularnewline
\hline 
$3$ & $s_{(2,1)}+4s_{(3)}$\tabularnewline
\hline 
$4$ & $2s_{(2,2)}+4s_{(3,1)}+7s_{(4)}$\tabularnewline
\hline 
$5$ & $s_{(2,2,1)}+s_{(3,1,1)}+9s_{(3,2)}+9s_{(4,1)}+11s_{(5)}$\tabularnewline
\hline 
$6$ & $2s_{(2,2,2)}+8s_{(3,2,1)}+12s_{(3,3)}+3s_{(4,1,1)}+21s_{(4,2)}+16s_{(5,1)}+16s_{(6)}$\tabularnewline
\hline 
$7$ & $s_{(2,2,2,1)}+2s_{(3,2,1,1)}+14s_{(3,2,2)}+18s_{(3,3,1)}+21s_{(4,2,1)}$\\
$\vphantom{a^{a^{a}}}\qquad\qquad\qquad\qquad+34s_{(4,3)}+6s_{(5,1,1)}+38s_{(5,2)}+25s_{(6,1)}+22s_{(7)}$\tabularnewline
\hline 
\end{tabular}
\par\end{centering}
\caption{\label{tb-schur4}Schur expansions of $Q_{n}^{(1)}(1234)$ for $0\leq n\leq7$.}
\end{table}

\bibliographystyle{plain}
\addcontentsline{toc}{section}{\refname}\bibliography{bibliography}

\begin{thebibliography}{10}

\bibitem{Baril2017}
Jean-Luc Baril and Sergey Kirgizov.
\newblock The pure descent statistic on permutations.
\newblock {\em Discrete Math.}, 340(10):2550--2558, 2017.

\bibitem{Barnabei2010}
Marilena Barnabei, Flavio Bonetti, and Matteo Silimbani.
\newblock The descent statistic on {$123$}-avoiding permutations.
\newblock {\em S\'{e}m. Lothar. Combin.}, 63:Art. B63a, 8 pp., 2010.

\bibitem{Braenden2008}
Petter Br{\"a}nd{\'e}n.
\newblock Actions on permutations and unimodality of descent polynomials.
\newblock {\em European J. Combin.}, 29(2):514--531, 2008.

\bibitem{Braenden2015}
Petter Br\"{a}nd\'{e}n.
\newblock Unimodality, log-concavity, real-rootedness and beyond.
\newblock In {\em Handbook of {E}numerative {C}ombinatorics}, Discrete Math.
  Appl. (Boca Raton), pages 437--483. CRC Press, Boca Raton, FL, 2015.

\bibitem{Bukata2019}
Michael Bukata, Ryan Kulwicki, Nicholas Lewandowski, Lara Pudwell, Jacob Roth,
  and Teresa Wheeland.
\newblock Distributions of statistics over pattern-avoiding permutations.
\newblock {\em J. Integer Seq.}, 22(2):Art. 19.2.6, 22 pp., 2019.

\bibitem{Carlitz1977}
L.~Carlitz and Richard Scoville.
\newblock Zero-one sequences and {F}ibonacci numbers.
\newblock {\em Fibonacci Quart.}, 15(3):246--254, 1977.

\bibitem{Cheng2013}
Szu-En Cheng, Sergi Elizalde, Anisse Kasraoui, and Bruce~E. Sagan.
\newblock Inversion polynomials for 321-avoiding permutations.
\newblock {\em Discrete Math.}, 313(22):2552--2565, 2013.

\bibitem{Dokos2012}
Theodore Dokos, Tim Dwyer, Bryan~P. Johnson, Bruce~E. Sagan, and Kimberly
  Selsor.
\newblock Permutation patterns and statistics.
\newblock {\em Discrete Math.}, 312(18):2760--2775, 2012.

\bibitem{Elizalde2011}
Sergi Elizalde.
\newblock Fixed points and excedances in restricted permutations.
\newblock {\em Electron. J. Combin.}, 18(2):Paper 29, 17, 2011.

\bibitem{Erdoes1935}
P.~Erd\"{o}s and G.~Szekeres.
\newblock A combinatorial problem in geometry.
\newblock {\em Compositio Math.}, 2:463--470, 1935.

\bibitem{Flajolet2009}
Philippe Flajolet and Robert Sedgewick.
\newblock {\em Analytic {C}ombinatorics}.
\newblock Cambridge University Press, 2009.

\bibitem{Foata1970}
Dominique Foata and Marcel-P. Sch\"utzenberger.
\newblock {\em Th\'eorie {G}\'eom\'etrique des {P}olyn\^omes {E}ul\'eriens},
  volume Vol. 138 of {\em Lecture Notes in Mathematics}.
\newblock Springer-Verlag, Berlin-New York, 1970.

\bibitem{Hamaker2020}
Zachary Hamaker, Brendan Pawlowski, and Bruce~E. Sagan.
\newblock Pattern avoidance and quasisymmetric functions.
\newblock {\em Algebr. Comb.}, 3(2):365--388, 2020.

\bibitem{Krattenthaler2001}
C.~Krattenthaler.
\newblock Permutations with restricted patterns and {D}yck paths.
\newblock {\em Adv. in Appl. Math.}, 27(2-3):510--530, 2001.
\newblock Special issue in honor of Dominique Foata's 65th birthday
  (Philadelphia, PA, 2000).

\bibitem{Liu2007}
Lily~L. Liu and Yi~Wang.
\newblock A unified approach to polynomial sequences with only real zeros.
\newblock {\em Adv. in Appl. Math.}, 38(4):542--560, 2007.

\bibitem{macmahon}
P.~A. MacMahon.
\newblock {\em Combinatory {A}nalysis}.
\newblock Two volumes (bound as one). Chelsea Publishing Co., New York, 1960.
\newblock Originally published in two volumes by Cambridge University Press,
  1915--1916.

\bibitem{Petersen2015}
T.~Kyle Petersen.
\newblock {\em Eulerian Numbers}.
\newblock Birkh\"auser/Springer, New York, 2015.

\bibitem{Petersen}
T.~Kyle Petersen and Yan Zhuang.
\newblock Zig-zag {E}ulerian polynomials.
\newblock Preprint, \url{https://arxiv.org/abs/2403.07181}, 2024.

\bibitem{Riordan1958}
John Riordan.
\newblock {\em An {I}ntroduction to {C}ombinatorial {A}nalysis}.
\newblock Wiley Publications in Mathematical Statistics. John Wiley \& Sons,
  Inc., New York; Chapman \& Hall, Ltd., London, 1958.

\bibitem{Robertson2002}
Aaron Robertson, Dan Saracino, and Doron Zeilberger.
\newblock Refined restricted permutations.
\newblock {\em Ann. Comb.}, 6(3-4):427--444, 2002.

\bibitem{Sagan2020b}
Bruce~E. Sagan.
\newblock {\em Combinatorics: {T}he {A}rt of {C}ounting}, volume 210 of {\em
  Graduate Studies in Mathematics}.
\newblock American Mathematical Society, Providence, RI, [2020] \copyright
  2020.

\bibitem{Sapounakis2006}
A.~Sapounakis, I.~Tasoulas, and P.~Tsikouras.
\newblock Dyck path statistics.
\newblock {\em WSEAS Trans. Math.}, 5(5):459--464, 2006.

\bibitem{Shareshian2010}
John Shareshian and Michelle~L. Wachs.
\newblock Eulerian quasisymmetric functions.
\newblock {\em Adv. Math.}, 225(6):2921--2966, 2010.

\bibitem{Shareshian2016}
John Shareshian and Michelle~L. Wachs.
\newblock Chromatic quasisymmetric functions.
\newblock {\em Adv. Math.}, 295:497--551, 2016.

\bibitem{Simion1994}
Rodica Simion.
\newblock Combinatorial statistics on noncrossing partitions.
\newblock {\em J. Combin. Theory Ser. A}, 66(2):270--301, 1994.

\bibitem{Simion1985}
Rodica Simion and Frank~W. Schmidt.
\newblock Restricted permutations.
\newblock {\em European J. Combin.}, 6(4):383--406, 1985.

\bibitem{oeis}
N.~J.~A. Sloane.
\newblock The {O}n-{L}ine {E}ncyclopedia of {I}nteger {S}equences.
\newblock Published electronically at \url{https://oeis.org}.

\bibitem{Stanley2024}
Richard~P. Stanley.
\newblock {\em Enumerative {C}ombinatorics. {V}ol. 2}, volume 208 of {\em
  Cambridge Studies in Advanced Mathematics}.
\newblock Cambridge University Press, Cambridge, second edition, [2024]
  \copyright 2024.
\newblock With an appendix by Sergey Fomin.

\bibitem{Stembridge1997}
John~R. Stembridge.
\newblock Enriched ${P}$-partitions.
\newblock {\em Trans. Amer. Math. Soc.}, 349(2):763--788, 1997.

\bibitem{Zhuang2017}
Yan Zhuang.
\newblock Eulerian polynomials and descent statistics.
\newblock {\em Adv. in Appl. Math.}, 90:86--144, 2017.

\end{thebibliography}

\appendix

\section{Appendix: Proof of Lemmas \ref{lem:G}--\ref{lem:FG}} \label{s-B}

The purpose of this appendix is to prove Lemmas \ref{lem:G} and \ref{lem:FG}, which we shall restate below for the reader's convenience.

Let us recall the relevant definitions for Lemma \ref{lem:G}. Given a Dyck path $\nu \in \mathcal{D}_m$, we write $U_i$ for the $i$th $U$ step of $\nu$ and $D_i$ for its $i$th $D$ step. Then, $\con(\nu)$ is the number of $i\in[m-1]$ such that $D_i$ and $D_{i+1}$ are consecutive but $U_i$ and $U_{i+1}$ are not. The generating function $G(s,t,z)$ is given by
$$G(s,t,z)=\sum_{m=0}^\infty \sum_{\nu\in\mathcal{D}_m} s^{\pk(\nu)} t^{\con(\nu)} z^m$$
where $\pk(\nu)$ is the number of peaks of $\nu$ and $\con(\nu)$ is defined as above.

\begin{lemG}
    We have $$G(s,t,z)=\frac{1-(1+s-2t)z-\sqrt{1-2(1+s)z+((1+s)^2-4st)z^2}}{2tz}.$$
\end{lemG}

Our proof of Lemma \ref{lem:G} will utilize a bijection from Dyck paths to ``$2$-Motzkin paths''. A \textit{$2$-Motzkin path} of length $n$ is a lattice path in $\mathbb{Z}^2$ from $(0,0)$ to $(n,0)$ that never goes below the $x$-axis and consists of up steps $u=(1,1)$, down steps $d=(1,-1)$, and two types of horizontal steps $(1,0)$ denoted by $h_0$ and $h_1$. Let $\MM_{n}$ denote the set of $2$-Motzkin paths of length $n$.

For $m\ge1$, we define $\psi\colon\mathcal{D}_m\to\MM_{m-1}$ in the following way. Take $\nu\in\mathcal{D}_m$, color in red the $U$ and $D$ steps of $\nu$ that belong to peaks, and color the remaining steps in blue.
Note that $D_1$ and $U_m$ are always red. Then $\psi(\nu)$ is the $2$-Motzkin path of length $m-1$ whose $i$th step is
$$
\begin{cases}
u,&\text{if $U_i$ is blue and $D_{i+1}$ is red,}\\
d,&\text{if $U_i$ is red and $D_{i+1}$ is blue},\\
h_0,&\text{if both $U_i$ and $D_{i+1}$ are blue,}\\
h_1,&\text{if both $U_i$ and $D_{i+1}$ are red,}
\end{cases}
$$
for all $i\in[m-1]$. See Figure \ref{f-2Motzbij} for an example.
It can be shown that $\psi$ is a bijection, and it is in fact equivalent to the bijection $\mu$ from \cite[Section 2]{Cheng2013}---specifically, we have $\mu=\psi\circ\chi$, where $\chi$ is the bijection from 321-avoiding permutations to Dyck paths defined in Section \ref{ss-321}.

\begin{figure}
\begin{center}
\begin{tikzpicture}[scale=0.5] 
\draw [line width=0] (4,2); 
\draw[pathcolorlight] (0,0) -- (16,0); 
\drawlinedotsbluept{2,3}{0,1};
\drawlinedotsbluept{5,6,7}{1,2,3};
\drawlinedotsbluept{11,12}{3,2};
\drawlinedotsbluept{14,15,16}{2,1,0};
\drawlinedotsredpt{0,1,2}{0,1,0};
\drawlinedotsredpt{3,4,5}{1,2,1};
\drawlinedotsredpt{7,8,9,10,11}{3,4,3,4,3};
\drawlinedotsredpt{12,13,14}{2,3,2};
\end{tikzpicture}
\begin{tikzpicture}[scale=1]
\draw [line width=0] (0,0); 
\node (0) at (0,1) {$\overset{\psi}{\longmapsto}$};
\end{tikzpicture}
\begin{tikzpicture}[scale=0.75] 
\draw [line width=0] (4,2); 
\draw[pathcolorlight] (0,0) -- (7,0);
\drawlinedots{0,1,2,3,4,5,6,7}{0,0,1,1,1,2,1,0};
\node (1) at (0.5,0.4) {{\footnotesize $h_1$}};
\node (2) at (2.5,1.4) {{\footnotesize $h_1$}};
\node (3) at (3.5,1.4) {{\footnotesize $h_0$}};
\end{tikzpicture}
\end{center}
\vspace{-10bp}
\caption{\label{f-2Motzbij}An example illustrating the bijection $\psi$. In the Dyck path, we color the peaks \textcolor{red}{red} and the other steps \textcolor{blue}{blue}.}
\end{figure}
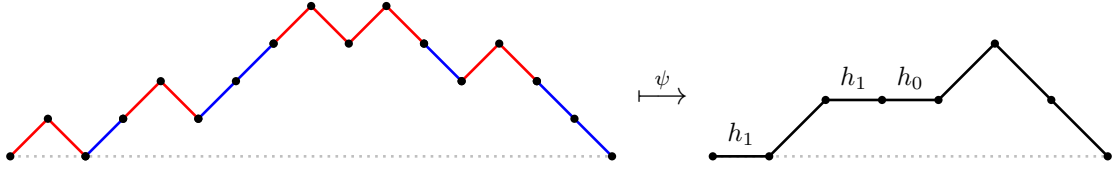

\begin{proof}[Proof of Lemma \ref{lem:G}]
We begin by determining how the Dyck path statistics $\pk$ and $\con$ translate to $2$-Motzkin paths via the bijection $\psi$. Given a $2$-Motzkin path $\alpha$, denote by $u(\alpha)$, $d(\alpha)$, $h_0(\alpha)$, and $h_1(\alpha)$ the number of $u$, $d$, $h_0$, and $h_1$ steps of $\alpha$, respectively. 

Fix a Dyck path $\nu$, and let $\alpha = \psi(\nu)$. First, $\pk(\nu)$ equals the number of red $U$ steps of $\nu$. Since red $U$ steps (except the last one) become $d$ and $h_1$ steps of $\alpha$, we have $\pk(\nu)=d(\alpha)+h_1(\alpha)+1$.
Second, the condition that $D_i$ and $D_{i+1}$ are consecutive, but $U_i$ and $U_{i+1}$ are not, is equivalent to saying that $U_i$ is red and $D_{i+1}$ is blue. It follows that $\con(\nu)=d(\alpha)$. Thus, we have
\begin{equation}\label{eq:tildeG} G(s,t,z)=1+sz\tilde{G}(s,t,z)\end{equation} where
$$\tilde{G}(s,t,z)\coloneqq\sum_{n=0}^\infty \sum_{\alpha\in\MM_n} s^{d(\alpha)+h_1(\alpha)} t^{d(\alpha)} z^n.$$

An equation for $\tilde{G}(s,t,z)$ can be obtained using the standard decomposition of $2$-Motzkin paths: any nonempty $2$-Motzkin path is either of the form $h_1\alpha'$, $h_2\alpha'$, or $u\alpha'd\alpha''$, where $\alpha'$ and $\alpha''$ are arbitrary $2$-Motzkin paths. It follows that  
$$\tilde{G}(s,t,z)=1+(1+s)z\tilde{G}(s,t,z)+tsz^2\tilde{G}(s,t,z)^2.$$
Solving for $\tilde{G}(s,t,z)$ and using \eqref{eq:tildeG}, we obtain the desired formula for $G(s,t,z)$.
\end{proof}

For Lemma \ref{lem:FG}, let us recall the definition of the generating function $F(u,v,w,x)$. As before, we color the peaks of a Dyck path red and its remaining steps blue. Then
$$F(u,v,w,x) = \sum_{n=0}^\infty \sum_{\mu\in\mathcal{D}_n} u^{\hibasc(\mu)} v^{\lobasc(\mu)} w^{\ini_{UU}(\mu)} x^n$$
where
\begin{itemize}
\item $\hibasc(\mu)$ is the number of peaks of $\mu$, other than the first one, that are not immediately preceded by another peak;
\item $\lobasc(\mu)$ is the number of indices $\ell$ such that the $\ell$th and $(\ell+1)$th blue $D$ steps of $\mu$ are consecutive but the $\ell$th and $(\ell+1)$th blue $U$ steps are not; and,
\item $\ini_{UU}(\mu)$ is equal to $1$ if $\mu$ begins with a $UU$-factor, and is equal to $0$ otherwise.
\end{itemize}

\begin{lemFG}
    We have
    $$F(u,v,w,x)=\frac{1}{u}\left(w+\frac{ux}{1-x}\right)\left(G(s(u,v,x),t(u,v,x),z(u,v,x))-1\right)+\frac{1}{1-x}$$
    where 
    \begin{align*}
    a(u,v,x) &\coloneqq 1+\frac{ux}{1-x}\left(1+v+\frac{ux}{1-x}\right), & s(u,v,x) &\coloneqq \frac{\frac{ux}{1-x}\left(1+\frac{ux}{1-x}\right)}{a(u,v,x)},\\
    t(u,v,x) &\coloneqq \frac{\left(1+\frac{ux}{1-x}\right)\left(v+\frac{ux}{1-x}\right)}{a(u,v,x)}, \quad \text{and} &  z(u,v,z) &\coloneqq a(u,v,x)\,x.
    \end{align*}
\end{lemFG}

\begin{proof}

    Given $\mu\in\mathcal{D}_n$, we color its steps in red and blue as before, and let $r$ be its number of peaks. Restricting to the blue steps of this path, we obtain a new Dyck path $\nu\in\mathcal{D}_{n-r}$ that we call the {\em core} of $\mu$, and we write $\core(\mu)=\nu$ to indicate this. Conversely, given any Dyck path $\nu$ consisting of blue steps, we can obtain all the Dyck paths $\mu$ for which $\core(\mu)=\nu$ by inserting red peaks $UD$. At any vertex of $\nu$, we can insert any number of peaks, except that this number cannot be zero at vertices that are in the middle of a peak of $\nu$---this is because all peaks in the resulting path $\mu$ must be red.

    Now, fix $\nu\in\mathcal{D}_m$ with $\pk(\nu)=p$ and $\con(\nu)=c$, and suppose that $m\ge1$. We will compute the generating function 
    \[
    F_{\nu}(u,v,w,x)=\sum_{n=0}^{\infty}\sum_{\substack{\mu\in\mathcal{D}_{n}\\
    \core(\mu)=\nu}}u^{\hibasc(\mu)}v^{\lobasc(\mu)}w^{\ini_{UU}(\mu)}x^{n}
    \]
    which represents the contribution of all Dyck paths with core $\nu$ to the full generating function $F(u,v,w,x)$. For $\hibasc(\nu)$, it will be convenient to instead consider $\hibasc(\nu)+1$ so that it counts all peaks not immediately preceded by another peak, including the first peak; we will divide by $u$ at the end to account for this difference.

    Let us analyze the insertion of a sequence of peaks at each of the $2m+1$ vertices of $\nu$. As before, for $i\in[m]$, let $U_i$ and $D_i$ denote the $i$th $U$ step and the $i$th $D$ step of $\nu$, respectively. Let us first consider inserting sequences of peaks right before $U_{i+1}$ and right before $D_{i+1}$, for each $i\in[m-1]$. These indices fall into three categories, each of which affects the generating function $F_{\nu}(u,v,w,x)$ in a different way:
\begin{itemize}
    \item Since $\con(\nu)=c$, there are $c$ values of $i\in[m-1]$ for which $D_i$ and $D_{i+1}$ are consecutive but $U_i$ and $U_{i+1}$ are not. In this case, inserting a sequence of peaks before $U_{i+1}$ contributes $1+\frac{ux}{1-x}$, since the number of peaks of $\mu$ not immediately preceded by a peak goes up by one if at least one peak is inserted. Inserting peaks before $D_{i+1}$ contributes $v+\frac{ux}{1-x}$, where the term $v$ records that $\lobasc(\mu)$ increases by one when no peaks are inserted, since $D_i$ and $D_{i+1}$ remain consecutive blue steps in $\mu$. Therefore, insertions for these $c$ values of $i$ contribute $$\left(1+\frac{ux}{1-x}\right)^c\left(v+\frac{ux}{1-x}\right)^c.$$
    \item Among the remaining $m-1-c$ values of $i\in[m-1]$, there are $p-1$ for which $D_{i+1}$ belongs to a peak of $\nu$ (since $\pk(\nu)=p$ and $D_1$ always belongs to a peak), which is equivalent to saying that $D_i$ and $D_{i+1}$ are not consecutive. 
    Again, inserting peaks before $U_{i+1}$ contributes $1+\frac{ux}{1-x}$, regardless of whether $U_i$ and $U_{i+1}$ are consecutive.
    Inserting peaks before $D_{i+1}$ contributes $\frac{ux}{1-x}$, since we are required to insert at least one peak. Therefore, insertions for these $p-1$ values of $i$ contribute $$\left(1+\frac{ux}{1-x}\right)^{p-1}\left(\frac{ux}{1-x}\right)^{p-1}.$$
    \item For the remaining $m-c-p$ values of $i\in[m-1]$, we have that $D_i$ and $D_{i+1}$ are consecutive, and so are $U_i$ and $U_{i+1}$.
    If no peak is inserted before $D_{i+1}$, then the insertion of a sequence of peaks before $U_{i+1}$ contributes $1+\frac{uvx}{1-x}$; this is because when at least one peak is inserted before $U_{i+1}$, the blue steps $U_i$ and $U_{i+1}$ are no longer consecutive in $\mu$, whereas the blue steps $D_i$ and $D_{i+1}$ are. If one or more peaks are inserted before $D_{i+1}$, these contribute $\frac{ux}{1-x}$, whereas the sequence of peaks inserted before $U_{i+1}$ now contributes $1+\frac{ux}{1-x}$. Therefore, insertions for these $m-c-p$ values of $i$ contribute $$\left(1+\frac{uvx}{1-x}+\frac{ux}{1-x}\left(1+\frac{ux}{1-x}\right)\right)^{m-c-p}.$$
\end{itemize}    

The above cases cover insertions of peaks before right before any step, except for $U_1$ and $D_1$. When inserting peaks before $U_1$, the value of $\ini_{UU}(\mu)$ will be $1$ if no peak is inserted, and $0$ otherwise, whereas the number of peaks not preceded by a peak increases by one if at least one peak is inserted. This produces a contribution of $w+\frac{ux}{1-x}$. Inserting peaks before $D_1$ contributes $\frac{ux}{1-x}$, since $D_1$ belongs to a peak of $\nu$, so at least one new (red) peak must be inserted.
Finally, we can also insert any number of peaks at the very end of $\nu$, which contributes $1+\frac{ux}{1-x}$.

Multiplying all these contributions, and dividing by $u$ so that its exponent does not include the first peak, we have that for each $\nu\in\mathcal{D}_m$ (where $m\ge1$) with $\pk(\nu)=p$ and $\con(\nu)=d$---which contributes $s^pt^cz^m$ to the generating function $G(s,t,z)$---the set of Dyck paths with core $\nu$ contributes
\begin{multline*}
F_\nu(u,v,w,x)=\frac{1}{u}\left(w+\frac{ux}{1-x}\right)\left(1+\frac{ux}{1-x}\right)^c\left(v+\frac{ux}{1-x}\right)^c\\
\left(1+\frac{ux}{1-x}\right)^{p}\left(\frac{ux}{1-x}\right)^{p} \left(1+\frac{ux}{1-x}\left(1+v+\frac{ux}{1-x}\right)\right)^{m-c-p}
x^m
\end{multline*}
to the generating function $F(u,v,w,x)$. Thus, the contribution of all Dyck paths with a nonempty core---i.e., the sum of $F_\nu(u,v,w,x)$ over all nonempty $\nu$---is obtained from $G(s,t,z)-1$ by making the substitutions described above and multiplying by $\frac{1}{u}\left(w+\frac{ux}{1-x}\right)$. Finally, Dyck paths with an empty core are those consisting of a sequence of peaks $UD$, and these contribute $1/(1-x)$ to the generating function $F(u,v,w,x)$.
\end{proof}

\end{document}